\documentclass[10pt,reqno,a4]{amsart}

\usepackage[latin1]{inputenc}
\usepackage{amssymb,amsmath,amscd,amsfonts,color, amsthm}
\title[On bilinear forms based on the resolvent of random matrices]{On bilinear forms based on the resolvent of 
large random matrices}

\author[Hachem et al.]{Walid Hachem, Philippe Loubaton, Jamal Najim
  and Pascal Vallet} \date{\today}

\newtheorem{theo}{Theorem}[section]

\newtheorem{lemma}[theo]{Lemma}
\newtheorem{coro}[theo]{Corollary}

\newtheorem{prop}[theo]{Proposition}
\newtheorem{assump}{Assumption A-\hspace{-0.15cm}}
\newcommand{\R}{\mathbb{R}}
\newcommand{\C}{\mathbb{C}}
\newcommand{\E}{\mathbb{E}}
\newcommand{\bdm}{\begin{displaymath}}
\newcommand{\edm}{\end{displaymath}}
\newcommand{\bea}{\begin{eqnarray*}}
\newcommand{\eea}{\end{eqnarray*}}

\newcommand{\Cplus}{\mathbb{C}^+}
\newcommand{\Rplus}{\mathbb{R}^+}

\newcommand{\im}{\mathrm{Im}}
\newcommand{\bs}{\boldsymbol}

\newcommand{\ii}{\mathbf{i}}
\numberwithin{equation}{section}
\theoremstyle{remark}
\newtheorem{rem}{Remark}[section]

\newcommand{\dzz}{\boldsymbol{\delta}_z}
\newcommand{\tr}{\mathrm{Tr}\,}
\newcommand{\cvgP}[1]{\xrightarrow[#1]{\mathcal P}}
\newcommand{\cvgD}{\xrightarrow[]{\mathcal D}}

\DeclareMathOperator*{\Imm}{Im}
\DeclareMathOperator*{\Real}{Re}
\DeclareMathOperator*{\dist}{dist}

\begin{document}
\bibliographystyle{plain}
\begin{abstract}
  Consider a $N\times n$ non-centered matrix $\Sigma_n$ with a
  separable variance profile:
 $$
\Sigma_n= \frac{D_n^{1/2} X_n \tilde D_n^{1/2}}{\sqrt{n}} + A_n\ .
$$ 
Matrices $D_n$ and $\tilde D_n$ are non-negative deterministic
diagonal, while matrix $A_n$ is deterministic, and $X_n$ is a random
matrix with complex independent and identically distributed random
variables, each with mean zero and variance one. Denote by $Q_n(z)$
the resolvent associated to $\Sigma_n \Sigma_n^*$, i.e.
$$
Q_n(z)=\left( \Sigma_n \Sigma_n^* -zI_N\right)^{-1}\ .
$$
Given two sequences of deterministic vectors $(u_n)$ and $(v_n)$ with
bounded Euclidean norms, we study the limiting behavior of
the random bilinear form:
$$
u_n^* Q_n(z) v_n\ ,\qquad    \forall z\in \C - \R^+\ ,
$$ 
as the dimensions of matrix $\Sigma_n$ go to infinity at the same
pace. Such quantities arise in the study of functionals of $\Sigma_n
\Sigma_n^*$ which do not only depend on the eigenvalues of $\Sigma_n
\Sigma_n^*$, and are pivotal in the study of problems related to
non-centered Gram matrices such as central limit theorems,
individual entries of the resolvent, and eigenvalue separation. 
\end{abstract}

\maketitle

\section{Introduction} 
\subsection*{The model}
Consider a $N\times n$ random matrix $\Sigma_n=(\xi_{ij}^{n})$ given by:
\begin{equation}\label{eq:model-sigma}
\Sigma_n= \frac {D_n^{\frac 12} {X_n} \tilde
  D_n^{\frac 12}}{\sqrt{n}} + A_n  \stackrel{\triangle}= Y_n +A_n\ ,
\end{equation}
where $D_n$ and $\tilde D_n$ are respectively $N\times N$ and $n\times
n$ non-negative deterministic diagonal matrices.  The entries of
matrices $(X_n)$, $(X_{ij}^{n}\ ; i,j,n)$ are complex, independent and
identically distributed (i.i.d.) with mean 0 and variance 1, 
and $A_n=(a_{ij}^{n})$ is a
deterministic $N\times n$ matrix whose spectral norm is bounded in $n$.

The purpose of this article is to study bilinear forms based on the
resolvent $Q_n(z)$ of matrix $\Sigma_n \Sigma_n^*$, where $\Sigma_n^*$
stands for the hermitian adjoint of $\Sigma_n$:
$$
Q_n(z)=\left( \Sigma_n \Sigma_n^* -zI_N\right)^{-1}\ ,
$$
as the dimensions $N$ and $n$ grow to infinity at the same pace, that is:
\begin{equation}\label{eq:asymptotic}
  0\quad <\quad \liminf \frac Nn\quad  \le \quad \limsup \frac Nn \quad <\quad \infty\ ,
\end{equation}
a condition that will be referred to as $N,n\rightarrow \infty$ in the
sequel.

A lot of attention has been devoted to the study of quadratic forms
$y^* A y$, where $y=n^{-1/2} (X_1,\cdots X_n)^T$, the $X_i$'s being
i.i.d., and $A$ is a matrix independent from $y$. It is well-known, at
least since Marcenko and Pastur's seminal paper \cite[Lemma
1]{MarPas67} (see also \cite[Lemma 2.7]{BaiSil98}) that under fairly
general conditions, $y^* A y \sim_{\infty} n^{-1} \tr A$. 

Such a result is of constant use in the study of centered random
matrices, as it allows to describe the behavior of the Stieltjes
transform associated to the spectral measure (empirical distribution
of the eigenvalues) of the matrix under investigation, see for
instance \cite{Sil95}, \cite{SilBai95}, \cite{HLN06,HLN07}, etc. 
Indeed, the Stieltjes transform of the spectral measure writes: 
$$
f_n(z)= \frac 1N  \tr Q_n(z) = \frac 1N \sum_{i=1}^N [Q_n(z)]_{ii}(z)\ ,
$$
where the $[Q_n(z)]_{ii}$'s denote the diagonal elements of the resolvent.
Denote by $\tilde\eta_i$ the $i$th row of $\Sigma_n$ and by
$\Sigma_{n,i}$ matrix $\Sigma_n$ when row $\tilde\eta_i$ has been
removed, then the matrix inversion lemma yields the following
expression:
$$
[Q_n(z)]_{ii} = -\frac 1{z\left(1+\tilde\eta_i 
(\Sigma_{n,i}^* \Sigma_{n,i} -zI)^{-1}\tilde\eta_i^* \right)}\ .
$$
In the case where $\Sigma_n= n^{-1/2} X_n$, the quadratic form that
appears in the previous expression can be handled by the 
aforementioned results. However, if $\Sigma_n$ is non-centered and given by
\eqref{eq:model-sigma}, then the quadratic form writes:
$$
\tilde\eta_i \tilde Q_i(z)\tilde\eta_i^* = \tilde y_i\tilde Q_i(z)\tilde y_i^* 
+ \tilde a_i\tilde Q_i(z)\tilde y_i^*
+ \tilde y_i\tilde Q_i(z)\tilde a_i^*
+ \tilde a_i\tilde Q_i(z)\tilde a_i^* \ ,
$$
where $\tilde Q_i(z)= (\Sigma_{n,i}^* \Sigma_{n,i} -zI)^{-1}$, and 
$\tilde y_i$ and $\tilde a_i$ are the $i$th rows of matrices $Y_n$ and $A_n$. 
The first term can be 
handled as in the centered case, the second and third terms go 
to zero; however, the fourth term involves a quadratic form  
$\tilde a_i\tilde Q_i(z)\tilde a_i^*$ based on deterministic vectors.

It is of interest to notice that, due to some fortunate cancellation,
the particular study of bilinear forms of the type $u_n^* Q_n(z)v_n$
or their analogues of the type $\tilde u_n \tilde Q_n(z) \tilde v_n^*$
can be circumvented to establish first order results for non-centered
random matrices (see for instance \cite{DozSil07a}, \cite{HLN07}).
However, such a study has to be addressed for finer questions such as:
Asymptotic behavior of individual entries of the resolvent (see for
instance \cite[Eq. (2.16)]{ERD10pre} where such properties are established in
the centered Wigner case to describe fine properties of the spectrum)
, Central Limit Theorems \cite{kam-kha-hac-naj-elk-spawc10,HKNS11pre}, behavior
of the extreme eigenvalues of $\Sigma_n \Sigma_n^*$, behavior of the
eigenvalues and eigenvectors associated with finite rank perturbations
of $\Sigma_n \Sigma_n^*$ \cite{ben-rao-arxiv09}, behavior of
eigenvectors or projectors on eigenspaces of $Q(z)$ (see for instance
\cite{BMP07} in the context of sample covariance (centered) model),
etc.

In a more applied setting, functionals based on individual entries of
the resolvent \cite{artigue-loubaton-2011} naturally arise in the
field of wireless communication (see for instance Section
\ref{sec:appli1}). Moreover, the asymptotic study of the quadratic
forms $u_n^* Q_n(z) u_n$ is important in statistical inference
problems.  In the non-correlated case (where $D_n=I_N$ and $\tilde
D_n=I_n$), it is proved in \cite{vallet2010sub} how such quadratic
forms yield consistent estimates of projectors on the subspace
orthogonal to the column space of $A_n$ in the Gaussian case (see also
Section \ref{sec:appli2}). Such projectors form the basis of MUSIC
algorithm, very popular in the field of antenna array processing. A
similar approach has been developed in \cite{Mes08-it},
\cite{Mes08-sp} for sample covariance matrix models.

It is the purpose of this article to provide a quantitative
description of the limiting behavior of the bilinear form $u_n^*
Q_n(z) v_n$, where $u_n$ and $v_n$ are deterministic, as the
dimensions of $\Sigma_n$ go to infinity as indicated in
\eqref{eq:asymptotic}.

\subsection*{Assumptions, fundamental equations, deterministic equivalents} 
Formal
assumptions for the model are stated below, where $\|\cdot\|$ either
denotes the Euclidean norm of a vector or the spectral norm of a
matrix.
\begin{assump}
\label{ass:X}
The random variables $(X_{ij}^n\ ;\ 1\le i\le N,\,1\le j\le n\,,\, n\ge1)$ 
are complex, independent and identically
distributed. They satisfy $\E X_{ij}^n = 0$ and $\E|X_{ij}^n|^2=1$. 
% $$
% \E X_{ij}^n = \E(X_{ij}^n)^2=0, \quad\mathrm{and} \quad 
% \E|X_{ij}^n|^2=1 \quad \ .
% $$
\end{assump}

\begin{assump}
\label{ass:A} The family of deterministic $N\times n$ matrices $(A_n,n\ge 1)$ 
is bounded for the spectral norm as $N,n\rightarrow\infty$:
$$
\boldsymbol{a_{\max}}= \sup_{n\ge 1} \|A_n\| <\infty\ .
$$
\end{assump}
Notice that this assumption implies in particular that the Euclidean norm 
of any row or column of $\|A_n\|$ is uniformly bounded in $N,n$. 

\begin{assump}\label{ass:D} The families of real deterministic $N\times N$ and $n\times n$ 
  matrices $(D_n)$ and $(\tilde D_n)$ are diagonal with non-negative
  diagonal elements, and are bounded for the spectral norm as
  $N,n\rightarrow\infty$:
$$
\boldsymbol{d_{\max}}= \sup_{n\ge 1} \|D_n\| <\infty\quad\textrm{and}\quad
\boldsymbol{\tilde d_{\max}}= \sup_{n\ge 1} \|\tilde D_n\| <\infty\ .
$$  
Moreover, 
$$
\boldsymbol{d_{\min}} = \inf_{N} \frac 1N \tr D_n >0\quad \textrm{and}\quad 
\boldsymbol{\tilde d_{\min}} = \inf_{n} \frac 1n \tr \tilde D_n >0\ .
$$
\end{assump}

We collect here results from \cite{HLN07}. 

The following system of equations:
\begin{equation}\label{eq:fundamental}
\left\{
  \begin{array}{ccc}
\delta(z) & = & \frac 1n \mathrm{Tr}\, D_n
\left( -z(I_N+\tilde \delta(z) D_n )I_N + A_n\left( I_n + \delta(z) \tilde D_n\right)^{-1} A_n^*
\right)^{-1}
\\
\tilde \delta(z) & = & \frac 1n \mathrm{Tr}\, \tilde D_n
\left( -z(I_n +\delta(z)\tilde D_n) + A^*_n\left( I_N + \tilde \delta(z) D_n\right)^{-1}  A_n
\right)^{-1}
\end{array}
\right. ,\quad z\in \mathbb{C}- \mathbb{R}^+
\end{equation}
admits a unique solution $(\delta, \tilde \delta)$ in the class of Stieltjes
transforms of nonnegative measures\footnote{In fact, $\delta$ and $\tilde \delta$ are the
  Stieltjes transforms of measures with respective total mass $n^{-1}\tr D_n$
and $n^{-1}\tr \tilde D_n$.} with support in $\mathbb{R}^+$. Matrices $T_n(z)$ and $\tilde T_n(z)$ defined by 
\begin{equation}\label{eq:def-T}
\left\{
\begin{array}{ccc}
T_n(z) & = & 
\left( -z(I_N+\tilde \delta(z) D_n ) + A_n\left( I_n + \delta(z) \tilde D_n\right)^{-1} A_n^*
\right)^{-1}
\\
\tilde T_n(z) & = & 
\left( -z(I_n +\delta(z)\tilde D_n) + A^*_n\left( I_N + \tilde \delta(z) D_n\right)^{-1}  A_n
\right)^{-1}
\end{array}
\right. 
\end{equation}
are approximations of the resolvent $Q_n(z)$
and the co-resolvent $\tilde Q_n(z) = (\Sigma_n^* \Sigma_n -zI_N)^{-1}$ in the
sense that ($\xrightarrow[]{a.s.}$ stands for the almost sure
convergence):
$$
\frac 1N \mathrm{Tr}\, \left( Q_n(z) -T_n(z)\right) \xrightarrow[N,n\rightarrow \infty]{a.s.} 0\ ,
$$
which readily gives a deterministic approximation of the Stieltjes
transform $N^{-1} \mathrm{Tr}\, Q_n(z)$ of the spectral measure of
$\Sigma_n \Sigma_n^*$ in terms of $T_n$ (and similarly for $\tilde
Q_n$ and $\tilde T_n$). Matrices $T_n$ and $\tilde T_n$ will play a
fundamental role in the sequel.

\subsection*{Nice constants and nice polynomials} By nice
constants, we mean positive constants which depend upon the limiting quantities
$\boldsymbol{d_{\min}}$, $\boldsymbol{\tilde d_{\min}}$,
$\boldsymbol{d_{\max}}$, $\boldsymbol{\tilde d_{\max}}$,
$\boldsymbol{a_{\max}}$, $\liminf \frac Nn$ and $\limsup \frac Nn$ but
are independent from $n$ and $N$.  Nice polynomials are polynomials
with fixed degree (which is a nice constant) and
with non-negative coefficients, each of them being a nice constant.
Further dependencies are indicated if needed.

\subsection*{Statement of the main result} 
%Let $(a_n)$ and $(b_n)$
%be complex sequences, then by $a_n={\mathcal O}(b_n)$, it is meant
%that there exists a constant $K$ independent of $n$ such that for $n$
%large enough, $a_n\le K b_n$. If $z\in \C- \R^+$, denote by:
%$$
%{\bf d}_z=\mathrm{dist}(z,\R^+)\ .
%$$
Let $\dzz$ be the distance between the point $z\in \C$ 
and the real nonnegative axis $\R^+$:
\begin{equation}\label{def:dzz}
\dzz \ =\ \dist(z,\R^+) \ .
\end{equation}  
% \begin{equation}\label{eq:delta}
% \dzz=
% \left\{
% \begin{array}{ll} 
% |\Imm(z)|& \textrm{if}\ z\in \Cplus\cup \C^- ,\\
% -z & \textrm{if}\ z\in (-\infty,0)\ .
% \end{array}
% \right.
% \end{equation} 
Here is the main result of the paper:

\begin{theo}
\label{th:main} 
  Assume that $N,n\rightarrow \infty$ and that assumptions
  A-\ref{ass:X}, A-\ref{ass:A} and A-\ref{ass:D} hold true. Assume
  moreover that there exists an integer $p\ge 1$ such that
  $\sup_n \mathbb{E} |X_{ij}^n|^{8p}<\infty$ and let $(u_n)$ and $(v_n)$ be
  sequences of $N\times 1$ deterministic vectors. Then, for every $z\in
  \C-\R^+$,
\begin{equation}
\label{ineq-main-th} 
  \E \left| u_n^* \left( Q_n(z) -T_n(z)\right) v_n\right|^{2p} \le
\frac{1}{n^p} \Phi_p(|z|)\Psi_p\left( \frac 1{\dzz}\right)\, \|u_n\|^{2p} \|v_n\|^{2p} , 
\end{equation} 
where $\Phi_p$ and $\Psi_p$ are nice polynomials depending on $p$ but 
not on $(u_n)$ neither on $(v_n)$.
\end{theo}

\begin{rem}
  Apart from providing the convergence speed ${\mathcal O}(n^{-p})$,
  inequality \eqref{ineq-main-th} provides a fine control of
  the behavior of $\E| u^*(Q-T)v |^{2p}$ when $z$ is near the real
  axis. Such a control should be helpful for studying the behavior of
  the extreme eigenvalues of $\Sigma_n \Sigma_n^*$ along the lines of
  \cite{BaiSil98} and \cite{bai-sil-exact99}.
\end{rem}

\begin{rem} {\em Influence of the eigenvectors of $AA^*$ on the
    limiting behavior of $u^* Q u$.} Consider a matrix $\Sigma$ with
  no variance profile ($D=I_N,\ \tilde D= I_n$) and let $T$ be given by \eqref{eq:def-T}. Matrix $T$ writes
  in this case:
$$
T=\left( -z(1+\tilde \delta)I +\frac{AA^*}{1+\delta} \right)^{-1}\ .
$$
Denote by $V \Delta V^*$ the spectral decomposition of $A A^*$, and by $T_{\Delta}$:
$$
T_{\Delta} =\left( -z(1+\tilde \delta)I +\frac{\Delta}{1+\delta} \right)^{-1}\ .
$$
Obviously, $T= V T_{\Delta} V^*$ and by Theorem \ref{th:main}, $u^* Q
u - u^* VT_{\Delta}V^* u\rightarrow 0\ .$ Clearly, the limiting
behavior of $u^* Q u$ not only depends on the spectrum (matrix
$\Delta$) of $AA^*$ but also on its eigenvectors (matrix $V$).
\end{rem}

\subsection*{Contents}
In Section \ref{sec:applications}, we describe two important
motivations from electrical engineering. In Section \ref{sec:misc}, we
set up the notations, state intermediate results among which Lemma
\ref{lemma:QiaaQi}, which is the cornerstone of the paper. Loosely
speaking, this lemma whose idea can be found in the work of Girko
\cite{gir-sara} states that quantities such as
$$
\sum_{i=1}^n u^* Q_i a_i a_i^* Q_i u
$$ 
are bounded. This control turns out to be central to take
into account Assumption A-\ref{ass:A}. 
An intermediate deterministic matrix $R_n$ is introduced and the proof of Theorem \ref{th:main}
is outlined. Basically, the quantity of interest $u^* (Q-T)v$ is split into three parts:
$$
u^* (Q-T)v = u^* (Q-\E Q)v +u^* (\E Q-R)v +u^* (R-T)v\ , 
$$
each being studied separately.

In Section \ref{sec:proof1}, the proof of estimate of $u^* (Q-\E Q)v$
is established, based on a decomposition of $Q-\E Q$ as a sum of
martingale increments. Section \ref{sec:proof2} is devoted to the
proof of estimate of $u^* (\E Q-R)v$; and Section \ref{sec:proof3}, to
the proof of estimate of $u^* (R-T)v$.

\subsection*{Acknowledgment}
This work was partially supported by the Agence Nationale de la
Recherche (France), project SESAME n$^\circ$ANR-07-MDCO-012-01.

\section{Two applications to electrical engineering}\label{sec:applications}

Apart from the technical motivations already mentionned in the
introduction, Theorem \ref{th:main} has further applications in
electrical engineering. In this section, we present an application to
Multiple Input Multiple Output (MIMO) wireless communication systems,
and an application to statistical signal processing.

\subsection{Optimal precoder in MIMO systems}\label{sec:appli1}
A bi-correlated MIMO wireless Ricean channel is a $N \times n$ random matrix $H_n$ given by
$$
H_n = B_n + R_n^{1/2} \frac{V_n}{\sqrt{n}} \tilde{R}_n^{1/2}\ ,
$$
where $B_n$ is a deterministic matrix, $V_n$ is a standard complex
Gaussian matrix, and where $R_n$ and $\tilde{R}_n$ represent
deterministic positive $N \times N$ and $n \times n$ matrices. An
important related question is the determination of a precoder
maximizing the so-called capacity after mininum mean square error
detection (for more details on the application context, see
\cite{artigue-loubaton-2011}). Mathematically, this problem is
equivalent to the evaluation of a deterministic $N \times N$ matrix
$K_n$ maximizing the function ${\mathcal I}_{mmse}(K_n)$ defined on
the set of all complex valued $N \times N$ matrices by
\begin{equation}
  \label{eq:def-Cmmse}
  {\mathcal I}_{mmse}(K_n) = \mathbb{E} \sum_{j=1}^{N} \left[ \log \left( I + K_n H_n H_n^{*} K_n^* \right)^{-1}_{j,j} \right]
\end{equation}
under the constraint $\frac{1}{N} \mathrm{Tr}(K_n K_n^*) \leq a$ ($a >
0$). This optimization problem has no closed form solution and one must
rely on numerical computations. However, direct numerical attempts
such as optimization by steepest descent algorithms or Monte-Carlo
simulations to evaluate ${\mathcal I}_{mmse}(K_n)$ before optimization,
or any combination of these techniques, face major difficulties, among
which: Hardly tractable expressions for ${\mathcal I}_{mmse}(K_n)$,
and for its first and second derivatives, computationally intensive
algorithms when relying on Monte-Carlo simulations.

If $N$ and $n$ are large enough, an alternative approach consists in
deriving a large system approximation $\overline{{\mathcal
    I}}_{mmse}(K_n)$ of ${\mathcal I}_{mmse}(K_n)$ which, hopefully,
is simpler to optimize w.r.t. $K_n$. This idea has been successfully
developed in \cite{artigue-loubaton-2011}, in the case where $B_n=0$,
and in \cite{dumont2009capacity} in a slightly different context,
where the functional under consideration is the Shannon capacity
${\mathcal I}_{s}(K_n) = \mathbb{E} \log \mathrm{det}\left( I + K_n
  H_n H_n^{*} K_n^* \right)$.

In the remainder of this section, we consider the case where $B_n\neq
0$ and briefly indicate how Theorem \ref{th:main} comes into
play. First remark that for every deterministic matrix $K_n$, the
random matrix $K_n H_n$ writes: 
$$
K_n H_n = K_n B_n + (K_n R_n K_n^*)^{1/2} \frac{W_n}{\sqrt{n}} \tilde{R}_n^{1/2}
$$ 
where $W_n$ is standard Gaussian random matrix (notice that $(K_n R_n
K_n^*)^{-1/2} K_n R_n^{1/2}$ is unitary).

%We remark that matrix $(I + K_n H_n H_n^* K_n^*)^{-1}$ coincides with
%the resolvent of $ K_n H_n H_n^* K_n^*$ at $z = -1$. 

Using the eigenvalue/eigenvector decomposition of matrices $ K_n R_n
K_n^*$ and $\tilde{R}_n$, the unitary invariance of the canonical
equations \eqref{eq:fundamental}, and Theorem \ref{th:main}, one can
easily check that the diagonal entries of $(I + K_n H_n H_n^*
K_n^*)^{-1}$ have the same asymptotic behaviour (when $(n,N)
\rightarrow \infty$) as those of the deterministic matrix $T_n(K_n)$
defined by:
$$
T_n(K_n) = \left[ (I + \tilde{\delta}(K_n) K_n R_n K_n^*) + K_n B_n (I
  + \delta(K_n) \tilde{R}_n)^{-1} B_n^* K_n^* \right]^{-1}\ ,
$$
where $\delta(K_n)$ and $\tilde{\delta}(K_n)$ are the (unique) positive solutions of the system: 
\begin{equation}\label{eq:delta-appli}
\left\{
  \begin{array}{ccc}
\delta(K_n) & = & \frac{1}{n} \mathrm{Tr} K_n R_n K_n^*  \left[ (I + \tilde{\delta}(K_n) K_n R_n K_n^*) + K_n B_n (I + \delta(K_n) \tilde{R}_n)^{-1} B_n^* K_n^* \right]^{-1} \\
\tilde{\delta}(K_n) & = & \frac{1}{n} \mathrm{Tr} \tilde{R}_n  \left[ (I + \delta(K_n) \tilde{R}_n) + B_n^* K_n^*  (I + \tilde{\delta}(K_n)  K_n R_n K_n^*)^{-1} K_n B_n \right]^{-1}
\end{array}\right.\ .
\end{equation}
From this, it appears that ${\mathcal I}_{mmse}(K_n)$ can be
approximated by $\overline{{\mathcal I}}_{mmse}(K_n)$ given by:
$$
\overline{{\mathcal I}}_{mmse}(K_n) = \sum_{j=1}^{N} \log \left[ (I + \tilde{\delta}(K_n) K_n R_n K_n^*) + K_n B_n (I + \delta(K_n) \tilde{R}_n)^{-1} B_n^* K_n^* \right]_{j,j}^{-1}
$$
Although the values taken by function $K_n \rightarrow \overline{{\mathcal
    I}}_{mmse}(K_n)$ are defined through the implicit equations
\eqref{eq:delta-appli}, the first and second
derivatives of $\overline{{\mathcal I}}_{mmse}$ are easy to compute, and
the minimization of $\overline{{\mathcal I}}_{mmse}$ instead of ${\mathcal
  I}_{mmse}$ certainly leads to a computationally attractive
algorithm.  

A number of important related questions remain to be addressed,
e.g. the accuracy of the approximation $\overline{{\mathcal
    I}}_{mmse}(K_n)$, its impact on the error on the optimum solution,
the derivation of a more accurate approximation as in
\cite{artigue-loubaton-2011}, the development of an efficient
algorithm to compute the optimal $K_n^*$, etc.; however this already
underlines promising applications of Theorem \ref{th:main} in the
context of wireless communication.

\subsection{Statistical signal processing applications}\label{sec:appli2}

There are many important applications such as source localization
using antenna arrays, communication channel estimation, detection of
signals corrupted by additive noise, etc. where the observations are
stacked into a matrix $\Sigma_n$ given by \eqref{eq:model-sigma} in
which $A_n$ is a non observable deterministic matrix modelling the
information to be retrieved and where $Y_n$ is due to an additive
noise. It is therefore often relevant to estimate certain functionals
of matrix $A_n$ from $\Sigma_n$. In this section, we show how Theorem
\ref{th:main} is valuable and relevant in the context of subspace
estimators when $N$ and $n$ are of the same order of magnitude.

\subsubsection*{Subspace estimation} Assume that $\frac{N}{n} < 1$,
$D_n=I_N$ and $\tilde{D}_n=I_n$ (white noise); assume also that matrix
$\mathrm{Rank}(A_n) = r < N$ where $r$ may scale or not with the
dimensions $n$ and $N$.
% In the context of important applications such as source localization
% using antenna arrays, communication channel estimation, detection of
% signals corrupted by additive noise,....the observation is stacked
% into a matrix $\Sigma_n$ given by model (\ref{sigma= a + y}) in which
% matrix $A_n$ is a non observable deterministic matrix modelling the
% "useful information" and where matrix $Y_n$ is due to an additive
% noise. 
% Therefore, it can be relevant to estimate certain functionals
% of matrix $A_n$ from $\Sigma_n$. In this paragraph, we show that
% Theorem \ref{} appears to be relevant in the context of the so-called
% subspace estimators when $N$ and $n$ are of the same order of
% magnitude. We assume that $\frac{N}{n} < 1$, matrices $D_n$ and
% $\tilde{D}_n$ are reduced to $I$, and that matrix $\mathrm{Rank}(A_n)
% = r < N$ where $r$ may scale or not with the dimensions $n$ and
% $N$.
Denote by $\Pi_n$ the orthogonal projection on the kernel of matrix
$A_n$. The subspace estimation methods are based on the estimation
of quadratic forms $u_n^* \Pi_n u_n$ where $(u_n)_{n \in \mathbb{N}}$
represents a sequence of unit norm deterministic $N$--dimensional
vectors.

If $N$ if fixed while $n \rightarrow +\infty$, it is well known that
$\| \Sigma_n \Sigma_n^* - (A_n A_n^* + I) \| \rightarrow 0$. Hence, if
$\check{\Pi}_n$ represents the orthogonal projection matrix on the
eigenspace associated to the $N-r$ smallest eigenvalues of $\Sigma_n
\Sigma_n^*$, then $\| \check{\Pi}_n - \Pi_n \| \rightarrow 0$ and thus 
\begin{equation}\label{eq:finite-N}
u_n^* \check{\Pi}_n u_n - u_n^* \Pi_n u_n \xrightarrow[n\rightarrow \infty,\, N\, \textrm{fixed}]{} 0\ .
\end{equation}
In order to model situations in which $n$ and $N$ are large
and of the same order of magnitude, it is relevant to look for
estimators consistent in the regime given by \eqref{eq:asymptotic}.
Unfortunately, \eqref{eq:finite-N} is no longer valid in this context.

\subsubsection*{An estimator for large $N,n$} The starting point
of the estimator proposed in \cite{vallet2010sub}, inspired by \cite{MES08}, is based on the
observation that $\Pi_n$ writes:
$$ 
\Pi_n = \frac{1}{2 i \pi} \int_{{\mathcal C}^{-}} \left( A_n A_n^* - \lambda I \right)^{-1} \, d\lambda\ ,
$$
where ${\mathcal C}^{-}$ is a clockwise oriented contour enclosing $0$
but not the non-zero eigenvalues of $A_n A_n^{*}$. In the white noise case, matrix $T_n(z)$ 
writes:
$$
T_n(z) = (1 +  \delta_n(z)) \left( A_n A_n^* - w_n(z) I \right)^{-1}  \ ,
$$
where $w_n(z)$ is the function defined by $w_n(z) = z( 1 +
\delta_n(z))( 1 + \tilde{\delta}_n(z))$. It is shown in
\cite{vallet2010sub} that (under additional assumptions) such a contour ${\mathcal C}^-$ 
is the image under $w_n$ of the boundary $\partial {\mathcal R}_y$ of the rectangle 
$
{\mathcal R}_y = \{ z = x + iv, x \in [x_{-}, x_{+}], |v| \leq y \}
$
for well-chosen $x_-$ and $x_+$. A simple change of variable argument therefore yields 
the following formula for $\Pi_n$:
$$
\Pi_n =  \frac{1}{2 i \pi} \int_{\partial {\mathcal R}_y^{-}} \left(  A_n A_n^* - w_n(z) I \right)^{-1}  w_n^{'}(z) dz = 
\frac{1}{2 i \pi} \int_{{\partial \mathcal R}_y^{-}} T_n(z)  \frac{w_n^{'}(z)}{1 +  \delta_n(z)} dz\ .
$$
Hence, $u_n^* \Pi_n u_n$ is given by:
\begin{equation}
\label{eq:expre-Pi}
u_n^* \Pi_n u_n=  \frac{1}{2 i \pi} \int_{{\partial \mathcal R}_y^{-}} u_n^* T_n(z) 
u_n \frac{w_n^{'}(z)}{1 +  \delta_n(z)} dz\ .
\end{equation}

Eq. \eqref{eq:expre-Pi} is particularly interesting because all the
terms in the integrand admit consistent estimators: Quantities
$\delta_n(z)$ and $\tilde{\delta}_n(z)$ can be estimated by
$\hat{\delta}_n(z) = \frac{1}{n} \mathrm{Tr}(Q_n(z))$ and
$\hat{\tilde{\delta}}_n(z) = \frac{1}{n} \mathrm{Tr}(\tilde{Q}_n(z))$,
$w_n^{'}(z)$ can be estimated by the derivative of $\hat{w}_n(z) = z
(1 + \hat{\delta}_n(z))( 1 + \hat{\tilde{\delta}}_n(z))$; finally,
Theorem \ref{th:main} implies that $ u_n^* Q_n(z) u_n - u_n^* T_n(z)
u_n \rightarrow 0$ for $N,n\rightarrow\infty$. A reasonnable estimator
for $\Pi_n$ should therefore be
\begin{equation}
\label{eq:expre-tildePi}
\hat{\Pi}_n =  \frac{1}{2 i \pi} \int_{{\partial \mathcal R}_y^{-}}  Q_n(z)  \frac{\hat{w}_n^{'}(z)}{1 + \hat{\delta}_n(z)} dz
\end{equation}
and it should be expected that $u_n^*\hat{\Pi}_n u_n - u_n^* \Pi_n u_n \rightarrow 0$ for $N,n\rightarrow\infty$.

% Roughly speaking, $\delta_n(z)$ and $\tilde{\delta}_n(z)$ can be estimated consistently by 
% $\hat{\delta}_n(z) = \frac{1}{n} \mathrm{Tr}(Q_n(z))$ because $\delta_n(z) - \frac{1}{n} \mathrm{Tr}(Q_n(z)) \rightarrow 0$. Similarly, it holds that 
%  $\hat{\tilde{\delta}}_n(z) =  \frac{1}{n} \mathrm{Tr}(\tilde{Q}_n(z))$ is consistent estimate of $\tilde{\delta}_n(z)$. Therefore, $w_n^{'}(z)$ can be
% estimated by the derivative of $\hat{w}_n(z) = z (1 + \hat{\delta}_n(z))( 1 +  \hat{\tilde{\delta}}_n(z))$. Finally, Theorem \ref{} implies that 
% $$
%  u_n^* Q_n(z) u_n - u_n^* T_n(z) u_n \rightarrow 0
% $$
% so that it is reasonable to expect that if $\tilde{\Pi}_n$ denotes the matrix defined by 
% \begin{equation}
% \label{eq:expre-tildePi}
% \tilde{\Pi}_n =  \frac{1}{2 i \pi} \int_{{\partial \mathcal R}_y^{-}}  Q_n(z)  \frac{\hat{w}_n^{'}(z)}{1 + \hat{\delta}_n(z)} dz
% \end{equation}
% then $u_n^*\tilde{\Pi}_n u_n - u_n^* \Pi_n u_n \rightarrow 0$ when
% $(N,n) \rightarrow \infty$ in such a way that $\frac{N}{n} \rightarrow
% c$.
\subsubsection*{Remaining mathematical issues}
The full definition of $\hat{\Pi}_n$ requires to prove that none of
the poles of the integrand of the r.h.s. of
\eqref{eq:expre-tildePi} can be equal to $x_{-}$ or
$x_{+}$. Otherwise, the mere definition of $\hat{\Pi}_n$ does not make
sense. This problem has been solved in the Gaussian case in
\cite{vallet2010sub}. In the non Gaussian case, partial results concerning ``no
eigenvalue separation for the signal plus noise model''
\cite{bai-silverstein-2011} together with Theorem \ref{th:main} tend
to indicate that the estimator $u_n^*\hat{\Pi}_n u_n$ is also
consistent.
  
% To check this, it is however necessary to verify that almost
% surely, for $n$ large enough, none of the poles of the integrand of
% the righthandside of (\ref{eq:expre-tildePi} can be equal to $x_{-}$
% or $x_{+}$. Otherwise, the definition of $\tilde{\Pi}_n$ does not make
% sense. In \cite{vallet2010sub}, this problem has been solved in the
% Gaussian case. Partial results in this direction have been recently
% established in \cite{bai-silverstein-2011} in the non Gaussian case.
% Theorem \ref{} tends to indicate that the estimator
% $u_n^*\tilde{\Pi}_n u_n$ is also consistent in the non Gaussian case.

\section{Notations, preliminary results and sketch of proof}\label{sec:misc}
\subsection{Notations}
\label{sec:notation}
The indicator function of the set ${\mathcal A}$ will be denoted by
${\bf 1}_{\mathcal A}(x)$, its cardinality by $\#{\mathcal A}$. Denote
by $a\wedge b= \inf(a,b)$ and by $a\vee b=\sup(a,b)$.  As usual,
$\Rplus = \{ x \in \R \ : \ x \geq 0 \}$ and $\Cplus = \{ z \in \C \ :
\ \im(z) > 0 \}$; similarly $\C^-=\{z \in \C \ : \ \im(z) < 0 \}$;
$\mathbf{i}=\sqrt{-1}$; if $z\in \mathbb{C}$, then $\bar{z}$ stands
for its complex conjugate.  Denote by $\cvgP{ }$ the convergence in
probability of random variables and by $\cvgD$ the convergence in
distribution of probability measures.  Denote by
$\mathrm{diag}(a_i;\,1\le i\le k)$ the $k\times k$ diagonal matrix
whose diagonal entries are the $a_i$'s.  Element $(i,j)$ of matrix $M$
will be either denoted $m_{ij}$ or $[M]_{ij}$ depending on the
notational context. if $M$ is a $n\times n$ square matrix,
$\mathrm{diag}(M)=\mathrm{diag}(m_{ii}; 1\le i\le n)$. Denote by $M^T$
the matrix transpose of $M$, by $M^*$ its Hermitian adjoint, by $\tr
(M)$ its trace and $\det(M)$ its determinant (if $M$ is square).  We
shall use Landau's notation: By $a_n={\mathcal O}(b_n)$, it is meant
that there exists a nice constant $K$ such that $|a_n|\le K|b_n|$ as
$N,n\rightarrow \infty$. Recall that when dealing with vectors,
$\|\cdot\|$ will refer to the Euclidean norm; in the case of matrices,
$\|\cdot\|$ will refer to the spectral norm.

Due to condition \eqref{eq:asymptotic}, we can assume (without loss of
generality) that there exist $0<\boldsymbol{\ell^-}\le
\boldsymbol{\ell^+}<\infty$ such that 
$$
\forall N,n\in \mathbb{N}^*,\qquad \boldsymbol{\ell^-}\quad \le \quad \frac Nn 
\quad \le\quad 
\boldsymbol{\ell^+}\ .
$$

We may drop occasionally subscripts and superscripts $n$ for readability.

Denote by $Y$ the $N\times n$ matrix $n^{-1/2} D^{1/2} X\tilde D^{1/2}
$; by $(\eta_j)$, $(a_j)$, $(x_j)$ and $(y_j)$ the columns of matrices
$\Sigma$, $A$, $X$ and $Y$. Denote by $\Sigma_j$, $A_j$ and $Y_j$, the
matrices $\Sigma$, $A$ and $Y$ where column $j$ has been removed. The
associated resolvent is $Q_j(z)=(\Sigma_j \Sigma_j^*
-zI_N)^{-1}$. Denote by $\mathbb{E}_j$ the conditional expectation
with respect to the $\sigma$-field ${\mathcal F}_j$ generated by the
vectors $(y_\ell,\, 1\le \ell\le j)$.  By convention,
$\mathbb{E}_0=\mathbb{E}$. Denote by $\mathbb{E}_{y_j}$ the
conditional expectation with respect to the $\sigma$-field generated
by the vectors $(y_\ell,\, \ell\neq j)$.

\subsection{Classical and useful results}

We remind here classical identities of constant use in the sequel. The
first one expresses the diagonal elements of the co-resolvent; the
other ones are based on low-rank perturbations of inverses (see for
instance \cite[Sec. 0.7.4]{HorJoh94}).
\subsubsection*{Diagonal elements of the co-resolvent; rank-one perturbation of the resolvent}
\begin{eqnarray}
  \tilde q_{jj}(z) &=& -\frac 1{z(1+\eta_j^* Q_j(z) \eta_j)}\ , \label{eq:diag-coresolvent}\\
  Q(z) &=& Q_j(z) -\frac{Q_j(z) \eta_j \eta_j^* Q_j(z)}{1+\eta_j^* Q_j \eta_j}\ ,\label{eq:woodbury}\\
  Q_j(z) &=& Q(z) + \frac{Q(z) \eta_j \eta_j^* Q(z)}{1- \eta_j^* Q \eta_j}\ ,\label{eq:woodbury2}\\
  1+\eta_j^* Q_j \eta_j &=& \frac 1{1-\eta_j^* Q \eta_j}\ .\label{eq:woodbury3}
\end{eqnarray}
A useful consequence of \eqref{eq:woodbury} is:
\begin{equation}
\eta_j^* Q(z) = \frac{\eta_j^* Q_j(z)}{1+\eta_j^* Q_j(z) \eta_j} = -z \tilde q_{jj}(z) \eta_j^* Q_j(z)\ .
\label{eq:woodbury4}
\end{equation}
Recall that $\dzz = \dist(z, \R^+)$. 
% $$
% \dzz=
% \left\{
% \begin{array}{ll} 
% |\Imm(z)|& \textrm{if}\ z\in \Cplus\cup \C^- ,\\
% -z & \textrm{if}\ z\in (-\infty,0)\ .
% \end{array}
% \right.
% $$
% $\forall z\in \C - \R^+$, $\dzz \le |z|$. 
Considering the eigenvalues of $Q(z)$ immediately yields
$
\|Q(z)\|  \le \dzz^{-1}\ .
$
Taking into account the fact that 
$$
- \frac 1{z(1+n^{-1} {\tilde d_j} \tr Q_j
+a_j^* Q_j a_j)}\quad \textrm{and}\quad -\frac 1{z(1+\eta_j^* Q_j \eta_j)}
$$ are
Stieltjes transforms of probability measures over $\R^+$, and based on
standard properties of Stieltjes transforms (see for instance
\cite[Proposition 2.2]{HLN07}), we readily obtain the following
estimates:
\begin{equation}\label{eq:property-ST}
\frac 1{\left| 1+\frac{ {\tilde d}_j}{n} \tr DQ_j +a_j^* Q_j a_j\right|}\le \frac{|z|}{\ \dzz}
\quad \textrm{and} \quad 
\frac 1{\left| 1+\eta_j^* Q_j \eta_j\right|}\le \frac{|z|}{\ \dzz}\ , \quad \forall z\in \C- \R^+\ .
\end{equation}
The following lemma describes the behavior of quadratic forms based
on random vectors (see for instance \cite[Lemma 2.7]{BaiSil98}).

\begin{lemma}
\label{lemma:approx-quadra}
Let $\bs{x}=(x_1,\cdots, x_n)$ be a $n \times 1$ vector where the
$x_i$'s are centered i.i.d.~complex random variables with unit variance;
consider $p\ge 2$ and assume that $\E|x_1|^{2p}<\infty$. Let
$M=(m_{ij})$ be a $n\times n$ complex matrix independent of $\bs{x}$.
Then there exists a 
constant $K_p$ such that
$$
\E\left|\bs{x}^* M \bs{x} -\tr M\right|^p
\le K_p \left( \tr M M^*\right)^{p/2}\  . 
$$
Let $\bs{u} \in \C^n$ be deterministic, then
$
\E|\bs{x}^* \bs{u}|^p ={\mathcal O} (\|\bs{u}\|^p)
$.
Moreover,
$
\E \| \bs{x}\|^p ={\mathcal O}(n^{p/2})\ .
$
\end{lemma} 
Note by $D=\mathrm{diag}(d_i\,;\ 1\le i\le N)$ and $\tilde D=\mathrm{diag}(\tilde d_i\,;\ 1\le i\le n)$.
Gathering the previous estimates yields the following useful corollary:
\begin{coro}\label{coro:quadra} Let $z\in \C - \R^+$, and let $p\ge 2$. Denote by $\Delta_j$ the quantity:
$$
\Delta_j = \eta_j^* Q_j \eta_j - \frac {\tilde d_j}n \tr D Q_j - a_j^* Q_j a_j\ .
$$
Then 
$$
\E_{y_j} \left| 
\Delta_j
\right|^p = {\mathcal O}\left( 
\frac 1{n^{p/2}\, \dzz^p}
\right)\ .
$$ 
\end{coro}

\begin{theo}[Burkholder inequality]\label{theo:burk}
  Let $(X_k)$ be a complex martingale difference sequence with respect
  to the filtration $({\mathcal F}_k)$. For every $p\ge 1$, there
  exists $K_p$ such that:
$$
\E \left| \sum_{k=1}^n X_k\right|^{2p} \le K_p \left( 
\E \left( \sum_{k=1}^n \E\left( |X_k|^2 \mid {\mathcal F}_{k-1} \right)
\right)^p+\sum_{k=1}^n \E |X_k|^{2p}
\right) \ . 
$$
\end{theo}

\subsubsection*{A result on holomorphic functions:} 
\begin{lemma}[Part of Schwarz's lemma, Th.12.2 in \cite{rud-rca}]
\label{lm-schwarz}
Let $f$ be an holomorphic function on the open unit disc $U$ such that 
$f(0) = 0$ and $\sup_{z\in U} | f(z) | \leq 1$. Then $| f(z) | \leq |z|$
for every $z \in U$. 
\end{lemma}

\subsubsection*{Rules about nice polynomials and nice constants} Some very simple rules of
calculus related to nice polynomials will be particularly helpful in 
the sequel: 

If $(\Phi_k, 1\le k\le K)$ and
$(\Psi_k, 1\le k\le K)$ are nice polynomials, then there exist nice
polynomials $\Phi$ and $\Psi$ such that:
\begin{equation}\label{eq:rule1}
\sum_{k=1}^K \Phi_k(x)\Psi_k(y) \le \Phi(x) \Psi(y)\quad \textrm{for} \quad x,y>0.
\end{equation} 
Take for instance $\Phi(x)= \sum_{k=1}^K \Phi_k(x)$ and $\Psi(x)= \sum_{k=1}^K \Psi_k(x)$.

If $\Phi_1$ and $\Psi_1$ are nice polynomials, then there exist nice
polynomials $\Phi$ and $\Psi$ such that:
\begin{equation}\label{eq:rule2}
\sqrt{\Phi_1(x) \Psi_1(y)} \le \Phi(x) \Psi (y)\quad \textrm{for} \quad x,y>0.
\end{equation}
Take for instance $\Phi= 2^{-1}(1+\Phi_1)$ and $\Psi= (1+\Psi_1)$ and note that:
$$
\sqrt{\Phi_1(x) \Psi_1(y)} \le \frac 12 (1+\Phi_1(x)\Psi_1(y))\le 
\frac{(1+\Phi_1(x))}2 (1+\Psi_1(y))\ .
$$

The values of nice constants or nice polynomials may change from line
to line within the proofs, the constant or the polynomial remaining
nice.

\subsection{Important estimates}

\begin{lemma}\label{lemma:QaaQ} Assume that the setting of Theorem \ref{th:main} holds
  true. Let $u$ be a deterministic complex $N\times 1$ vector. Then,
  for every $z\in \mathbb{C} - \mathbb{R}^+$, the following estimates
  hold true:
\begin{eqnarray}\label{eq:QaaQ}
  \E \left( \sum_{j=1}^n \E_{j-1} \left(u^* Q a_j a_j^* Q^* u\right)
\ \right)^p &\le & 
K_p  \frac {\|u\|^{2p}}{\dzz^{2p}}\ , \\
  \E \left( \sum_{j=1}^n \E_{j-1} \left(u^* Q \eta_j \eta_j^* Q^* u\right)
\ \right)^p &\le&  
\tilde K_p\frac {|z|^p\, \|u\|^{2p}}{\dzz^{2p}} \ ,\label{eq:QetaetaQ}
\end{eqnarray}
where $K_p$ and $\tilde K_p$ are nice constants depending on $p$ but not on $\|u\|$.
\end{lemma}

Proof of Lemma \ref{lemma:QaaQ} is postponed to Appendix \ref{app:proof-misc}.

\begin{lemma}\label{lemma:QiaaQi}
  Assume that the setting of Theorem \ref{th:main} holds true. Let $u$
  be a deterministic complex $N\times 1$ vector. Then, for every $z\in
  \mathbb{C} - \mathbb{R}^+$, the following estimates hold true:
\begin{eqnarray}
  \sum_{j=1}^n \E \left( u^* Q_j a_j a_j^* Q_j^* u\right)^2 &\le& 
\Phi(|z|) \Psi\left( \frac 1\dzz \right)\, \|u\|^4\ ,
\label{eq:QiaaQi1}\\
\E \left( 
\sum_{j=1}^n \E_{j-1} \left(u^* Q_j a_j a_j^* Q_j^* u\right)\ \right)^p &\le&
\tilde \Phi(|z|) \tilde \Psi \left( \frac 1\dzz \right)\, \|u\|^{2p} , \label{eq:QiaaQi2}
\end{eqnarray}
where $\Phi,\Psi,\tilde \Phi$ and $\tilde \Psi$ are nice polynomials not 
depending on $\|u\|$. 
\end{lemma}

Proof of Lemma \ref{lemma:QiaaQi} is postponed to Appendix \ref{app:proof-misc}.

In order to proceed, it is convenient to introduce the following intermediate quantities ($z\in \C-\R^+$):
\begin{eqnarray}
  \alpha_n(z) &=& \frac 1n \tr D_n \E Q_n(z),
\qquad \tilde \alpha_n(z) \quad=\quad  \frac 1n \tr \tilde D_n \E \tilde Q_n(z), \label{eq:def-alpha}\\
  R_n(z) & = & 
  \left( -z(I_N+\tilde \alpha(z) D_n )I_N + A_n\left( I_n + \alpha(z) \tilde D_n\right)^{-1} A_n^*
  \right)^{-1}\ ,\label{eq:def-R}\\
\tilde R_n(z) & = & 
\left( -z(I_n +\alpha(z)\tilde D_n) + A^*_n\left( I_N + \tilde \alpha(z) D_n\right)^{-1}  A_n
\right)^{-1}\ .\label{eq:def-Rtilde}
\end{eqnarray}
A slight modification of the proof of \cite[Proposition
5.1-(3)]{HLN07} yields the following estimates:
$$
\| R_n(z)\|\le \frac 1\dzz\ ,\quad \| \tilde R_n(z)\|\le \frac 1\dzz \quad \textrm{for}\ z\in \C-\R^+\ .
$$
The same estimates hold true for $\| T_n(z)\|$ and $\| \tilde T_n(z)\|$.
% \begin{lemma}\label{lemma:trace-estimate}
%   Assume that the setting of Theorem \ref{th:main} holds true and let
%   $(M_n)$ be a family of deterministic $N\times N$ matrices with
%   uniformly bounded spectral norm. Then,
% \begin{itemize}
% \item[(i)] For all $z\in \C - \R$,
% $$
% \left| \frac 1n \tr M_n \E Q_n(z) - \frac 1n \tr M_n R_n(z) \right| 
% \le \frac 1n \Phi(|z|) \Psi\left(\frac 1{|\mathrm{Im}(z)|}\right)\ ,
% $$
% where $\Phi$ and $\Psi$ are polynomials with nonnegative coefficients.
% \item[(ii)] For all $z=-x\in (-\infty,0)$,
% $$
% \left| \frac 1n \tr M_n \E Q_n(-x) - \frac 1n \tr M_n R_n(-x) \right| 
% \le \frac 1n \Phi(x) \Psi\left(\frac 1x\right)\ ,
% $$
% where $\Phi$ and $\Psi$ are polynomials with nonnegative coefficients.
% \end{itemize}

% \end{lemma}

\subsection{Main steps of the proof}
In order to prove Theorem \ref{th:main}, we split the quantity
of interest $u^* (Q-T) u$ into three parts:
$$
u^* (Q-T) v = u^* (Q-\E Q) v + u^* (\E Q-R) v + u^* (R-T) v\ ,
$$
and handle each term separately in the following propositions:

\begin{prop}\label{prop:intermediate1}
  Assume that the setting of Theorem \ref{th:main} holds true. Let
  $(u_n)$ and $(v_n)$ be sequences of $N\times 1$ deterministic
  vectors. Then, for every $z\in \C-\R^+$,
$$
  \E \left| u_n^* \left( Q_n(z) - \E Q_n(z)\right) v_n\right|^{2p} \le
\frac{1}{n^p} \Phi_p(|z|)\Psi_p\left( \frac 1{\dzz}\right)\, \|u_n\|^{2p} \|v_n\|^{2p} , 
$$
where $\Phi_p$ and $\Psi_p$ are nice polynomials depending on $p$ but not 
on $(u_n)$ nor on $(v_n)$.
\end{prop}

Proposition \ref{prop:intermediate1} is proved in Section \ref{sec:proof1}.

\begin{prop}\label{prop:intermediate2}
Assume that the setting of Theorem \ref{th:main} holds true. 
\begin{itemize}
\item[(i)] Let $(u_n)$
  and $(v_n)$ be sequences of $N\times 1$ deterministic vectors. Then, 
for every $z\in \C-\R^+$, 
$$
   \left| u_n^* \left( \E Q_n(z) -R_n(z)\right) v_n\right| \le
\frac{1}{\sqrt{n}} \Phi(|z|)\Psi\left( \frac 1{\dzz}\right)\, \|u_n\|\, \|v_n\|, 
$$
where $\Phi$ and $\Psi$ are nice polynomials, not depending on $(u_n)$ nor on 
$(v_n)$.\\
\item[(ii)] Let $M_n$ be a $N\times N$ deterministic 
  matrix. Then, for every $z\in \C - \R^+$,
$$
\left| \frac 1n \tr M_n \E Q_n(z) - \frac 1n \tr M_n R_n(z) \right| 
\le \frac 1n \Phi(|z|) \Psi\left(\frac 1{\dzz}\right)\, \|M_n\| ,
$$
where $\Phi$ and $\Psi$ are nice polynomials, not depending on $M_n$.
\end{itemize}
\end{prop}

Proposition \ref{prop:intermediate2}-(i) is proved in Section
\ref{sec:proof2}; proof of Proposition \ref{prop:intermediate2}-(ii)
is very similar and thus omitted.

\begin{prop}\label{prop:intermediate3}
Assume that the setting of Theorem \ref{th:main} holds true. Let $(u_n)$
  and $(v_n)$ be sequences of $N\times 1$ deterministic vectors. 

Then, for every $z\in \C-\R^+$, 
$$
\left| u_n^* \left( R_n(z) -T_n(z)\right) v_n\right| \le
\frac{1}{n} \Phi(|z|)\Psi\left( \frac 1{\dzz}\right)\, \|u_n\|\, \|v_n\| , 
$$
where $\Phi$ and $\Psi$ are nice polynomials, not depending on $(u_n)$ nor on 
$(v_n)$.
\end{prop}

Proposition \ref{prop:intermediate3} is proved in Section \ref{sec:proof3}.

Theorem \ref{th:main} is then easily proved using these three
propositions together with inequality $|x + y + z|^{2p} \le
K_p(|x|^{2p} + |y|^{2p} + |z|^{2p})$ and \eqref{eq:rule1}.

\section{Proof of Proposition \ref{prop:intermediate1}}\label{sec:proof1}

Recall the decomposition:
$$
u^*(Q-T)v= u^*(Q-\E Q)v+ u^*(\E Q-R)v+u^*(R-T)v\ .
$$

In this section, we establish the estimate:
\begin{equation}\label{eq:Q-EQ}
  \E \left| u^* \left( Q(z) -\E Q(z)\right) v\right|^{2p} \le
\frac{1}{n^p} \Phi_p(|z|)\Psi_p\left( \frac 1{\dzz}\right) \|u\|^{2p}\, \|v\|^{2p} \ ,\quad \forall z\in \C-\R^+\ . 
\end{equation}

\subsection{Reduction to unit vectors and quadratic forms}\label{sec:reduction}
Using a polarization identity, it is sufficient in order to establish
estimate \eqref{eq:Q-EQ} for the bilinear form $u^*(Q-\E Q)v$ to
establish the related estimate for the quadratic form $u^*(Q-\E Q)u$
and for unit vectors $\|u\|$ (just consider $u/\|u\|$ if necessary):
\begin{equation}\label{eq:Q-EQ-quadra}
  \E \left| u^* \left( Q(z) -\E Q(z)\right) u\right|^{2p} \le
\frac{1}{n^p} \Phi_p(|z|)\Psi_p\left( \frac 1{\dzz}\right)  \ . 
\end{equation}

% Assume first that estimate \eqref{eq:Q-EQ} holds true for
% deterministic vectors of norm one, then it holds true for any deterministic vector. 
% Indeed, just consider 
% $$
% \tilde u_n = \frac{u_n}{\| u_n\|} \quad \textrm{and} \quad \tilde v_n = \frac{v_n}{\| v_n\|}
% $$
% and use the bilinear property. It is therefore sufficient to establish
% \eqref{eq:Q-EQ} for unit vectors $u_n$ and $v_n$.

% Assume now that \eqref{eq:Q-EQ} holds true for 
% quadratic forms, that is:
% \begin{equation}\label{eq:Q-EQ-quadra}
%   \E \left| u^* \left( Q(z) -\E Q(z)\right) u\right|^{2p} \le
% \frac{1}{n^p} \Phi_p(|z|)\Psi_p\left( \frac 1{\dzz}\right) \|u\|^{4p} \ . 
% \end{equation}
% Let $u_n$ and $v_n$ be unit real vectors and write: 
% \begin{eqnarray*}
% 2 u^t (Q -T) v & =& (u+v)^t (Q -T) (u+v) - {\bf i} (u+ {\bf i}v)^* (Q -T) (u+{\bf i}v)\\
% && - (1-{\bf i})\left(u^t (Q -T) u + v^t (Q -T) v\right)\ .
% \end{eqnarray*}
% All the terms of the right hand side can be estimated with the help of
% \eqref{eq:Q-EQ-quadra}; hence, applying \eqref{eq:rule1} yields
% \eqref{eq:Q-EQ} for unit and real vectors. Generalization to complex
% unit vectors $u_n$ and $v_n$ is straightforward.

% In order to establish estimate \eqref{eq:Q-EQ} for the bilinear form
% $u^*(Q-\E Q)v$, it is therefore sufficient to establish estimate
% \eqref{eq:Q-EQ-quadra} for the quadratic form $u^*(Q-\E Q)u$ and for
% unit vectors $\|u\|$ (just consider $u/\|u\|$ if necessary).

\subsection{Martingale difference sequence and Burkholder inequality}
We first express the difference $u^* (Q -\E Q)u$ as the sum of
martingale difference sequences:
\begin{eqnarray*}
u^* (Q-\E Q)u &=& \sum_{j=1}^n (\E_j - \E_{j-1}) (u^* Q u)\quad =\quad \sum_{j=1}^n (\E_j - \E_{j-1}) (u^* (Q-Q_j) u)\\
&=&  - \sum_{j=1}^n (\E_j - \E_{j-1}) \left( \frac{u^* Q_j \eta_j^* \eta_j Q_j u}{1+\eta_j^* Q_j \eta_j}\right)
\quad \stackrel{\triangle}=\quad -\sum_{j=1}^n (\E_j - \E_{j-1}) \Gamma_j\ .
\end{eqnarray*}
One can easily check that $((\E_j - \E_{j-1}) \Gamma_j)$ is the sum of
a martingale difference sequence with respect to the filtration
$({\mathcal F}_j, j\le n)$; hence Burkholder's inequality yields:
\begin{multline}\label{eq:burholder}
\E \bigg| \sum_{j=1}^n (\E_j - \E_{j-1}) \Gamma_j 
\bigg|^{2p}\\ \le K 
\left(
\E \bigg(
\sum_{j=1}^n \E_{j-1} \left| (\E_j - \E_{j-1}) \Gamma_j \right|^2 \bigg)^p 
+
\sum_{j=1}^n \E \left| (\E_j - \E_{j-1}) \Gamma_j \right|^{2p}  
\right)\ .
\end{multline}
Recall the definition of $\Delta_j= \eta_j^* Q_j \eta_j - n^{-1}{\tilde d_j} \tr D Q_j - a_j^* Q_j a_j$.
In order to control the right-hand side of Burkholder's inequality, we
write $\Gamma_j$ as:
\begin{eqnarray*}
\Gamma_j &=& \frac{u^* Q_j \eta_j^* \eta_j Q_j u}{1+\eta_j^* Q_j \eta_j}
\quad =\quad \frac{u^* Q_j \eta_j^* \eta_j Q_j u}{1+\eta_j^* Q_j \eta_j}
\times \frac{1 + \frac{\tilde d_j}n \tr DQ_j +a_j^* Q_j a_j}
{1 + \frac{\tilde d_j}n \tr DQ_j +a_j^* Q_j a_j} \\
& = & \frac{u^* Q_j \eta_j^* \eta_j Q_j u}{1+\eta_j^* Q_j \eta_j} 
\times \frac{ 1+\eta_j^* Q_j \eta_j - \Delta_j}
{1 + \frac{\tilde d_j}n \tr DQ_j +a_j^* Q_j a_j}\quad \stackrel{\triangle}{=}\quad \Gamma_{1j} - \Gamma_{2j} \ ,
\end{eqnarray*}
where
\begin{equation}\label{eq:def-Gamma-split}
\Gamma_{1j}= \frac{u^* Q_j \eta_j \eta_j^* Q_j u}{1+\frac {\tilde d_j}n \tr DQ_j + a_j^* Q_j a_j}
\quad \textrm{and}\quad 
\Gamma_{2j}= \frac{\Gamma_j  \Delta_j}{1+\frac {\tilde d_j}n \tr DQ_j + a_j^* Q_j a_j}\ .
\end{equation}

In the following proposition, we establish relevant estimates.

\begin{prop}\label{prop:big-prop} Assume that the setting of Theorem 
  \ref{th:main} holds true. There exist nice polynomials $(\Phi_i, 1\le
  i\le 4)$ and $(\Psi_i, 1\le i\le 4)$ such that the following estimates 
hold true:
\begin{eqnarray}
  \E \bigg(
  \sum_{j=1}^n \E_{j-1} \left| (\E_j - \E_{j-1}) \Gamma_{1j} \right|^2 \bigg)^p 
&\le & \frac 1{n^p} \Phi_1(|z|) \Psi_1\left( \frac 1{\dzz}
  \right)\ ,\label{eq:bigprop1}\\
  \sum_{j=1}^n \E \left| (\E_j - \E_{j-1}) \Gamma_{1j} \right|^{2p}  
&\le & \frac 1{n^p}\Phi_2(|z|) \Psi_2\left( \frac 1{\dzz}
  \right)\ ,
\label{eq:bigprop2}\\ 
  \E \bigg(
  \sum_{j=1}^n \E_{j-1} \left| (\E_j - \E_{j-1}) \Gamma_{2j} \right|^2 \bigg)^p 
&\le & \frac 1{n^p}\Phi_3(|z|) \Psi_3\left( \frac 1{\dzz}
  \right)\ ,
\label{eq:bigprop3}\\
  \sum_{j=1}^n \E \left| (\E_j - \E_{j-1}) \Gamma_{2j} \right|^{2p}  
&=& \frac 1{n^p}\Phi_4(|z|) \Psi_4\left( \frac 1{\dzz}
  \right)
\ .\label{eq:bigprop4}
\end{eqnarray}
\end{prop}

It is now clear that the proof of Proposition \ref{prop:intermediate1} directly
follows from Burkholder's inequality together with the estimates of
Proposition \ref{prop:big-prop}. The rest of the section is devoted to
the proof of Proposition \ref{prop:big-prop}.

\subsection{Proof of Proposition \ref{prop:big-prop}: Estimates \eqref{eq:bigprop1} and
\eqref{eq:bigprop2}} 

%Estimates \eqref{eq:bigprop3} and
%\eqref{eq:bigprop4} can be proved similarly and are therefore omitted.

We split $\Gamma_{1j}$ as $\Gamma_{1j}= \chi_{1j}+\chi_{2j} + \chi_{3j}$, where:
\begin{eqnarray*}
\chi_{1j}&=& \frac{u^* Q_j y_j y_j^* Q_j u}{1+\frac{ {\tilde d}_j}{n} \tr DQ_j +a_j^* Q_j a_j}\ ,\\ 
\chi_{2j}&=& \frac{y_j^* Q_j u u^* Q_j a_j}{1+\frac{ {\tilde d}_j}{n} \tr DQ_j +a_j^* Q_j a_j}
+\frac{a_j^* Q_j u u^* Q_j y_j}{1+\frac{ {\tilde d}_j}{n} \tr DQ_j +a_j^* Q_j a_j}\ ,\\ 
\chi_{3j}&=& \frac{u^* Q_j a_j a_j^* Q_j u}{1+\frac{ {\tilde d}_j}{n} \tr DQ_j +a_j^* Q_j a_j}\ .\\ 
\end{eqnarray*}
Notice that $(\E_j -\E_{j-1})(\chi_{3j}) =0$, hence $\chi_{3j}$ will play no further role in the sequel.
As $Q_j$ is independent from column $y_j$, we have:
\begin{equation}\label{eq:partial-chi1-zero}
(\E_j -\E_{j-1})(\chi_{1j}) = \frac{\tilde d_j}n \E_j\left( 
\frac{x_j^* D^{1/2} Q_j u u^* Q_j D^{1/2} x_j - \tr D Q_j u u^* Q_j }
{1+\frac{ {\tilde d}_j}{n} \tr DQ_j +a_j^* Q_j a_j}
\right) \ ,
\end{equation}
and
\begin{eqnarray}
\E_{j-1} \left| (\E_j -\E_{j-1})(\chi_{1j})
\right|^2 &\stackrel{(a)}\le& \frac{\boldsymbol{\tilde d}^{\,2}_{\boldsymbol{\max}}}{n^2}\times 
\E_{j-1}\left| 
\frac{x_j^* D^{1/2} Q_j u u^* Q_j D^{1/2} x_j - \tr D Q_j u u^* Q_j }
{1+\frac{ {\tilde d}_j}{n} \tr DQ_j +a_j^* Q_j a_j}
\right|^2\ ,\nonumber \\
&\stackrel{(b)}\le& \frac{\boldsymbol{\tilde d}^{\,2}_{\boldsymbol{\max}}}{n^2}
\frac{|z|^2}{\dzz^2}\times 
\E_{j-1}\left[ \E_{y_j}\left| 
x_j^* D^{1/2} Q_j u u^* Q_j D^{1/2} x_j - \tr D Q_j u u^* Q_j 
\right|^2 \right]\nonumber \\
&\stackrel{(c)}\le& K \frac{\boldsymbol{\tilde d}^{\,2}_{\boldsymbol{\max}}}{n^2}
\frac{|z|^2}{\dzz^2}\times 
\E_{j-1}\left( \tr D^{1/2} Q_j u u^* Q_j D^{1/2} D^{1/2} Q^*_j u u^* Q^*_j D^{1/2}\right)\nonumber \\
&=& {\mathcal O}\left( 
\frac{|z|^2}{n^2\, \dzz^6}
\right)\ ,\label{eq:partial-chi1}
\end{eqnarray}
where $(a)$ follows from Jensen's inequality, $(b)$ from estimate \eqref{eq:property-ST}, and $(c)$ 
from Lemma \ref{lemma:approx-quadra}. Thus
\begin{equation}\label{eq:estimate-chi1}
\E\bigg( \sum_{j=1}^n \E_{j-1} \left| (\E_j -\E_{j-1})(\chi_{1j})
\right|^2\bigg)^p \quad = \quad {\mathcal O}\left( 
\frac{|z|^{2p}}{n^p\, \dzz^{6p}}
\right)\ .
\end{equation}
We now turn to the contribution of $\chi_{2j}$. Arguments similar as 
previously yield:
\begin{eqnarray}
\E_{j-1} \left| (\E_j -\E_{j-1})(\chi_{2j}) \right|^2 
\! \!  \! \!      
&=&  
 \E_{j-1} \left| \E_j \chi_{2j} \right|^2
\leq 
\E_{j-1} \left| \chi_{2j} \right|^2 \nonumber \\ 
&\le &
\frac{2}{n} 
\E_{j-1} \left( 
\left| \frac{x_j^* D^{1/2} Q_j u u^* Q_j a_j}{1+\frac{ {\tilde d}_j}{n} \tr DQ_j +a_j^* Q_j a_j} 
\right|^2
+\left| \frac{a_j^* Q_j u u^* Q_j D^{1/2} x_j}{1+\frac{ {\tilde d}_j}{n} \tr DQ_j +a_j^* Q_j a_j} 
\right|^2
\right)\ ,\nonumber \\
&\le & \frac{2}{n} \frac{|z|^2}{\dzz^2}  
\E_{j-1} \left( \E_{y_j}(x_j^* D^{1/2} Q_j u  u^* Q_j^* D^{1/2} x_j)\times u^* Q_j a_j a_j^* Q_j^* u
\right.\nonumber \\
&&\left. \qquad \qquad \qquad \qquad + \E_{y_j}(x_j^* D^{1/2} Q_j^* u  u^* Q_j D^{1/2} x_j)\times 
u^* Q_j^* a_j a_j^* Q_j u
\right) \ ,\nonumber \\
&\le& \frac{K}{n} \frac{|z|^2}{\dzz^4} 
\left(   \E_{j-1} \left( u^* Q_j^* a_j a_j^* Q_j u \right)
+ \E_{j-1} \left( u^* Q_j a_j a_j^* Q^*_j u
\right) \right)\ .\label{eq:partial-chi2}
\end{eqnarray}
Now, using Eq. \eqref{eq:QiaaQi2} in Lemma \ref{lemma:QiaaQi} yields:
\begin{equation}\label{eq:estimate-chi2}
\E\bigg( \sum_{j=1}^n \E_{j-1} \left| (\E_j -\E_{j-1})(\chi_{2j})
\right|^2\bigg)^p \quad \le \quad  \frac 1{n^p} \Phi(|z|) \Psi\left( \frac 1\dzz\right).
\end{equation}
Hence, gathering \eqref{eq:estimate-chi1} and \eqref{eq:estimate-chi2}
yields estimate \eqref{eq:bigprop1}. 

We now establish estimate \eqref{eq:bigprop2}. As previously, consider
identity \eqref{eq:partial-chi1-zero}; take it this time to the power
$p$. Using the same arguments as for \eqref{eq:partial-chi1}, we obtain:
$$
\E \left| (\E_j -\E_{j-1})(\chi_{1j})
\right|^{2p}
\quad = \quad 
{\mathcal O}\left( 
\frac{|z|^{2p}}{n^{2p}\, \dzz^{6p}}
\right)\ ,
$$
hence: 
\begin{equation}\label{eq:estimate-chi1-bis}
\E \sum_{j=1}^n  \left| (\E_j -\E_{j-1})(\chi_{1j})
\right|^{2p} \quad = \quad {\mathcal O}\left( 
\frac{|z|^{2p}}{n^{2p-1}\, \dzz^{6p}}
\right)\ .
\end{equation}
Similarly, using the same arguments as in \eqref{eq:partial-chi2}, together with elementary manipulations, we obtain:
$$
\E_{j-1} \left| (\E_j -\E_{j-1})(\chi_{2j}) \right|^{2p} \le \frac{K}{n^p} \frac{|z|^{2p}}{\dzz^{4p}} 
\left(   \E_{j-1} \left( u^* Q_j^* a_j a_j^* Q_j u \right)^p
+ \E_{j-1} \left( u^* Q_j a_j a_j^* Q^*_j u
\right)^p \right)\ . 
$$
Due to the rough estimate \eqref{eq:rough-estimate}, we obtain
\begin{eqnarray*}
\E \left| (\E_j -\E_{j-1})(\chi_{2j}) \right|^{2p} 
&\le & \frac{K}{n^p} \frac{|z|^{2p}}{\dzz^{6p-4}} 
\left(   \E \left( u^* Q_j^* a_j a_j^* Q_j u \right)^2
+ \E \left( u^* Q_j a_j a_j^* Q^*_j u
\right)^2 \right)\ , 
\end{eqnarray*} 
which after summation, and the estimate obtained in Lemma \ref{lemma:QiaaQi}, yields:
\begin{equation}\label{eq:estimate-chi2-bis}
\E \sum_{j=1}^n  \left| (\E_j -\E_{j-1})(\chi_{2j})
\right|^{2p} \quad \le \quad \frac 1{n^p} \Phi'(|z|) \Psi'\left( \frac 1\dzz\right)\ ,
\end{equation}
where $\Phi'$ and $\Psi'$ are nice polynomials. Gathering 
\eqref{eq:estimate-chi1-bis} and \eqref{eq:estimate-chi2-bis} yields 
estimate \eqref{eq:bigprop2}.

\subsection{Proof of Proposition \ref{prop:big-prop}: Estimates \eqref{eq:bigprop3} and \eqref{eq:bigprop4}}
We split $\Gamma_{2j}$ as $\Gamma_{2j}= \chi_{1j}+\chi_{2j} + \chi_{3j}$, where:
\begin{eqnarray*}
\chi_{1j}&=& \Delta_j \times \frac{u^* Q_j a_j a_j^* Q_j u}{(1+\eta_j^* Q_j \eta_j)
(1+\frac{ {\tilde d}_j}{n} \tr DQ_j +a_j^* Q_j a_j)}\ ,\\ 
\chi_{2j}&=& \Delta_j \times \frac{u^* Q_j y_j y_j^* Q_j u}{(1+\eta_j^* Q_j \eta_j)
(1+\frac{ {\tilde d}_j}{n} \tr DQ_j +a_j^* Q_j a_j)}\ ,\\
\chi_{3j}&=& \Delta_j \times \frac{u^* Q_j y_j a_j^* Q_j u + u^* Q_j a_j y_j^* Q_j u}{(1+\eta_j^* Q_j \eta_j)
(1+\frac{ {\tilde d}_j}{n} \tr DQ_j +a_j^* Q_j a_j)}\ .\\
\end{eqnarray*}

Consider first:
\begin{eqnarray*}
\lefteqn{ \E_{j-1} \left| (\E_j - \E_{j-1})(\chi_{1j})\right|^2 \quad \le \quad  2 \E_{j-1}|\chi_{1j}|^2}\\
&\stackrel{(a)}\le & \frac{K|z|^4}{\dzz^4} \E_{j-1} \left| 
u^* Q_j a_j a_j^* Q_j u \left( 
y_j^* Q_j y_j -n^{-1}\tilde d_j \tr DQ_j 
\right)
\right|^2\\
&& \quad \quad +   \frac{K|z|^4}{\dzz^4} \E_{j-1} \left| 
u^* Q_j a_j a_j^* Q_j u \left( 
y_j^* Q_j a_j + a_j^* Q_j y_j
\right)
\right|^2\ ,\\
&\stackrel{(b)}\le & \frac{K|z|^4}{n^2\, \dzz^6}\E_{j-1}\left[ u^* Q_j a_j a_j^* Q_j^* u\  
\E_{y_j}  \left| x_j^* D^{1/2} Q_j D^{1/2} x_j - \tr DQ_j \right|^2 \right]\\
&& \quad \quad + \frac{K|z|^4}{n\, \dzz^6}\E_{j-1}\left[u^* Q_j a_j a_j^* Q_j^* u\   
\E_{y_j}  (x_j^* D^{1/2} Q_j a_j a_j^* Q_j^* D^{1/2} x_j ) \right]\\
&& \quad \quad \quad \quad \quad + \frac{K|z|^4}{n\, \dzz^6}\E_{j-1}\left[ u^* Q_j a_j a_j^* Q_j^* u\   
\E_{y_j}  (x_j^* D^{1/2} Q^*_j a_j a_j^* Q_j D^{1/2} x_j) \right]\\
&\stackrel{(c)}\le& \frac{K|z|^4}{n\, \dzz^8} \E_{j-1} (u^* Q_j a_j a_j^* Q_j^* u)\ ,
\end{eqnarray*}
where $(a)$ follows from \eqref{eq:property-ST}, $(b)$ from the fact
that $|u^* Q_j a_j a_j^* Q_j u|\le K \dzz^{-2}$ and $|u^* Q_j a_j a_j^* Q_j^* u|\le K
\dzz^{-2}$, and $(c)$ from Lemma \ref{lemma:approx-quadra}. 
From this and Lemma \ref{lemma:QiaaQi}, we deduce that:
\begin{equation}
\E\left( \sum_{j=1}^n \E_{j-1} \left| (\E_j - \E_{j-1})(\chi_{1j})\right|^2
\right)^p 
\quad \le \quad \frac 1{n^p} \Phi(|z|) \Psi\left( \frac 1\dzz\right)\ .
\end{equation}
Consider now:
\begin{eqnarray*}
 \E_{j-1} \left| (\E_j - \E_{j-1})(\chi_{2j})\right|^2 \quad \stackrel{(a)}{\le} \  2 \E_{j-1}|\chi_{2j}|^2
\quad \stackrel{(b)}\le \ \frac{K|z|^4}{\dzz^4} \E_{j-1} \left| 
y_j^* Q_j u\right|^4 \left|  
\Delta_j
\right|^2
\quad \stackrel{(c)}\le\  \frac{K|z|^4}{n^3\, \dzz^{10}}\ ,
\end{eqnarray*}
where $(a)$ follows from the triangle and Jensen's inequality, $(b)$
from \eqref{eq:property-ST} and $(c)$ from Cauchy-Schwarz inequality,
Lemma \ref{lemma:approx-quadra} and Corollary \ref{coro:quadra}. \\ 
Hence, 
$$
\E \left( 
\sum_{j=1}^n  \E_{j-1} \left| (\E_j - \E_{j-1})(\chi_{2j})\right|^2  
\right)^p \quad 
=\quad {\mathcal O}\left(\frac{|z|^{4p}}{n^{2p}\, \dzz^{10p}}\right)\ .
$$
Similarly, one can prove that:
$$
\E \left( \sum_{j=1}^n  \E_{j-1} \left| (\E_j - \E_{j-1})(\chi_{3j})\right|^2  
\right)^p \quad 
=\quad {\mathcal O}\left(\frac{|z|^{4p}}{n^p\, \dzz^{10p}}\right)\ .
$$ 
Gathering the previous results yields the bound:
$$
\E \left( 
\sum_{j=1}^n  \E_{j-1} \left| (\E_j - \E_{j-1})(\Gamma_{2j})\right|^2  
\right)^p \quad 
\le \quad 
\frac 1{n^p} \Phi'(|z|) \Psi'\left(\frac 1\dzz\right)\ .
$$
We now evaluate the second part of Burkholder's inequality (and may
re-use notations $\Phi$ and $\Psi$ for different polynomials).
\begin{eqnarray*}
\sum_{j=1}^n \E \left| (\E_j - \E_{j-1})(\chi_{1j})\right|^{2p} 
&\le & \frac{K|z|^{4p}}{\dzz^{4p}} \sum_{j=1}^n \E \left( 
u^* Q_j a_j a_j^* Q^*_j u\right)^{2p} 
\E_{y_j}  \left| \Delta_j\right|^{2p} \\
&\stackrel{(a)}\le& \frac{K|z|^{4p}}{n^p\, \dzz^{6p}}   \sum_{j=1}^n \E \left( 
u^* Q_j a_j a_j^* Q^*_j u\right)^{2}\left( 
u^* Q_j a_j a_j^* Q^*_j u\right)^{2p-2}\\
&\le& \frac{K|z|^{4p}}{n^p\, \dzz^{10p-4}}   \sum_{j=1}^n \E \left( 
u^* Q_j a_j a_j^* Q^*_j u\right)^{2}\\
&\le & \frac 1{n^p} \Phi(|z|) \Psi\left( \frac 1\dzz \right)\ ,
\end{eqnarray*}
where $(a)$ follows from Corollary \ref{coro:quadra} and the last estimate, from 
Lemma \ref{lemma:QiaaQi}. Similar computations yield:
\begin{eqnarray*}
\sum_{j=1}^n \E \left| (\E_j - \E_{j-1})(\chi_{2j})\right|^{2p} &\le & \frac 1{n^{3p-1}} \Phi'(|z|) \Psi'\left( \frac 1\dzz \right)\ ,\\
\sum_{j=1}^n \E \left| (\E_j - \E_{j-1})(\chi_{3j})\right|^{2p} &\le &
\frac 1{n^{2p-1}} \Phi''(|z|) \Psi''\left( \frac 1\dzz \right) \ , 
\end{eqnarray*}
the first of these inequalities requiring the assumption 
$\sup_n \E | X_{ij}^n |^{8p} < \infty$ in the statement of Theorem 
\ref{th:main}. 
Gathering these three results yields:
$$
\sum_{j=1}^n \E \left| (\E_j - \E_{j-1})(\Gamma_{2j})\right|^{2p} \quad \le \quad 
\frac 1{n^p} \tilde \Phi(|z|) \tilde \Psi\left( \frac 1\dzz \right)\ ,
$$
and Proposition \ref{prop:big-prop} is proved.

\section{Proof of Proposition \ref{prop:intermediate2}}\label{sec:proof2}

% supprimé suite à review IHP

 Recall the decomposition:
 $$
 u^*(Q-T)v= u^*(Q-\E Q)v+ u^*(\E Q-R)v+u^*(R-T)v\ .
 $$
 In this section, we establish the estimate:
 $$
   \left| u^* \left( \E Q(z) - R(z)\right) v\right| \le
 \frac{1}{\sqrt n} \Phi(|z|)\Psi\left( \frac 1{\dzz}\right)\|u\|\, \|v\|\ , 
 $$

The argument referred to in Section \eqref{sec:reduction} still holds
true here; therefore it is sufficient to establish, for $z\in \C-\R^+$ and for a unit vector $u$:
\begin{equation}\label{eq:EQ-R}
  \left| u^* \left( \E Q(z) - R(z)\right) u\right| \le
\frac{1}{\sqrt n} \Phi(|z|)\Psi\left( \frac 1{\dzz}\right)\ , 
\end{equation}
Recalling that 
$
R= \left[ -z(I + \tilde \alpha D) + A (I + \alpha \tilde D)^{-1} A^* \right]^{-1}
$,
the resolvent identity yields:
\begin{eqnarray*}
  u^* (R-Q) u &=& u^* R(Q^{-1} - R ^{-1})Q u\ , \\
  &=& u^* R\left(\Sigma \Sigma^* -A (I + \alpha \tilde D)^{-1} A^*\right)Q u  + z\tilde \alpha u^* R D Q u\ ,\\
  &=& u^* R\left(\sum_{j=1}^n \eta_j \eta_j^* -\sum_{j=1}^n \frac{a_j a_j^* }{1+\alpha \tilde d_j}\right)Q u  
  + z\tilde \alpha u^* R D Q u\ ,\\
  &\stackrel{(a)}=& \sum_{j=1}^n \frac{ u^* R\eta_j \eta_j^* Q_j u}{1+\eta_j^* Q_j \eta_j} 
-  \sum_{j=1}^n \frac{ u^* R a_j a_j^* Q_j u}{1+\alpha \tilde d_j} \\
 && \qquad + \sum_{j=1}^n \frac{ u^* Ra_j a_j^* Q_j \eta_j \eta_j^* Q_j u}{(1+\eta_j^* Q_j \eta_j)(1+\alpha \tilde d_j)} 
  - \sum_{j=1}^n \frac{\tilde d_j}n \E\left( \frac 1{1+\eta_j^* Q_j \eta_j}\right) u^* R D Q u\ ,\\
&\stackrel{\triangle}{=}& \sum_{j=1}^n Z_j\ .
\end{eqnarray*}
where $(a)$ follows from \eqref{eq:woodbury} and \eqref{eq:woodbury4},
together with the mere definition of $\tilde \alpha$. 

As usual, we now write $\eta_j=y_j + a_j$, group the terms that
compensate one another and split $Z_j$ accordingly:
$$
Z_j=Z_{1j} + Z_{2j} + Z_{3j} + Z_{4j}\ , 
$$
where
\begin{eqnarray*}
  Z_{1j}&=& \frac{ y_j^* Q_j u u^* R y_j}{1+\eta_j^* Q_j \eta_j} 
  -\frac{ \tilde d_j}n \E\left( \frac 1{1+\eta_j^* Q_j \eta_j}\right) u^* R D Q u\ ,\\
  Z_{2j}&=& \frac{(\alpha \tilde d_j - y_j^* Q_j y_j) u^* R a_j a_j^* Q_j u}
{(1+\eta_j^* Q_j \eta_j)(1+\alpha \tilde d_j)}\ ,\\ 
  Z_{3j}&=& \frac{y_j^* Q_j u a_j^* Q_j y_j \times u^* R a_j} 
{(1+\eta_j^* Q_j \eta_j)(1+\alpha \tilde d_j)}\ ,\\ 
  Z_{4j}&=& \frac{u^* R y_j a_j^* Q_j u + u^* R a_j y_j^* Q_j u}{1+\eta_j^* Q_j \eta_j}\\
&& \quad - \frac{
y_j^* Q_j a_j u^* R a_j a_j^* Q_j u 
+ a_j^* Q_j y_j u^* R a_j a_j^* Q_j u 
}{(1+\eta_j^* Q_j \eta_j)(1+\alpha \tilde d_j)}\\
&&\quad \quad + \frac{
u^* R a_j a_j^* Q_j a_j y_j^* Q_j u 
+ u^* R a_j a_j^* Q_j y_j a_j ^* Q_j u 
}{(1+\eta_j^* Q_j \eta_j)(1+\alpha \tilde d_j)}
\end{eqnarray*}

Now, the estimate \eqref{eq:EQ-R} immediately follows from similar
estimates for the terms $\E \sum_{j=1}^n Z_{\ell j}$, $1\le
\ell\le 4$. The rest of the section is devoted to establish such estimates.

\subsection{Convergence to zero of $\sum_j \E Z_{1j}$} 
We have
\begin{eqnarray*}
\E Z_{1j} &=& 
\E\left( \frac{ y_j^* Q_j u u^* R y_j}{1+\eta_j^* Q_j \eta_j} \right)
-\frac{ \tilde d_j}n \E\left( \frac 1{1+\eta_j^* Q_j \eta_j}\right) \E( u^* R D Q u)\\
&=& \E \left[ 
\left( \frac{ y_j^* Q_j u u^* R y_j}{1+\eta_j^* Q_j \eta_j} \right)
 - \frac{\tilde d_j}n  \left( \frac{u^* RDQ_j u} {1+\eta_j^* Q_j \eta_j} \right) \right]\\
&& \quad + \frac{\tilde d_j}n \left[
\E \left( \frac{u^* R D Q_j u }{1+\eta_j^* Q_j \eta_j}\right) - \E \left( \frac 1{1+\eta_j^* Q_j \eta_j}
\right)\E (u^* RDQ_j u) 
\right] \\
&& \qquad + \frac{\tilde d_j}n \E \left( \frac 1{1+\eta_j^* Q_j \eta_j}
\right)\E (u^* RD(Q_j-Q) u) \\
&\stackrel{\triangle}= & \chi_{1j} + \chi_{2j }+ \chi_{3j}\ .
\end{eqnarray*}
We first handle $\chi_{ij}$. Recall that $\Delta_j= \eta_j^* Q_j
\eta_j - n^{-1}{\tilde d_j} \tr D Q_j - a_j^* Q_j a_j$. Since 
$\E_{y_j} (y_j^* Q_j u u^* R y_j) = \tilde d_j n^{-1} u^* R D Q_j u $, we get:
\begin{eqnarray*}
\chi_{1j} &= & 
\E \left[ 
\left( \frac{ y_j^* Q_j u u^* R y_j}{1+\eta_j^* Q_j \eta_j} \right)
 - \frac{\tilde d_j}n  \left( \frac{u^* RDQ_j u} {1+\eta_j^* Q_j \eta_j} \right) \right] \ ,\\
&=& 
\E \left[ 
\left( \frac 1{1+\eta_j^* Q_j \eta_j} - \frac 1{1+\frac{\tilde d_j} n \tr D Q_j + a_j^* Q_j a_j} \right)
\left(  y_j^* Q_j u u^* R y_j
 - \frac{\tilde d_j}n  ( u^* RDQ_j u) \right) \right]\ ,\\
&=& \E \left[
\Delta_j \frac{y_j^* Q_j u u^* R y_j
 - \frac{\tilde d_j}n  ( u^* RDQ_j u)}{(1+\eta_j^* Q_j \eta_j)(1+\frac{\tilde d_j} n \tr D Q_j + a_j^* Q_j a_j)}
\right]\ .
\end{eqnarray*}
Hence,
\begin{eqnarray*}
\left| \chi_{1j}\right| & \le & \frac{|z|^2}{\dzz^2} \sqrt{\E |\Delta_j|^2}\left[ 
\E \left|
 y_j^* Q_j u u^* R y_j
 - \frac{\tilde d_j}n  ( u^* RDQ_j u)
\right|^2
\right]^{1/2}\ , \\
&\le & \frac{|z|^2}{\dzz^2} \times \frac 1{\sqrt{n}\dzz} \times \frac 1 {n \dzz^2} \quad 
=\quad  {\mathcal O}\left( 
\frac{|z|^2 }{n^{3/2}\dzz^5}
\right)\ .
\end{eqnarray*}
Summing over $j$ yields the estimate $\sum_j |\chi_{1j} | = {\mathcal O} \left( 
|z|^2 n^{-1/2}\dzz^{-5}
\right)$.

We now handle $\chi_{2j}$. Using the inequality $\mathrm{cov}(XY)\le \sqrt{\mathrm{var}(X) \mathrm{var}(Y)}$, 
we get:
\begin{eqnarray*}
|\chi_{2j} | &\le & \frac Kn \frac {|z|}{\dzz} \sqrt{\E \left| u^* RD (Q_j -\E Q_j)u\right|^2} 
\end{eqnarray*}
Hence, applying Proposition \ref{prop:intermediate1} to $\left| u^* RD
  (Q_j -\E Q_j)u\right|^2$ and summing over $j$ yields the estimate
$\sum_j |\chi_{2j} | =n^{-1/2} \Phi(|z|)\Psi(\dzz^{-1})$.

Let us now handle the term $\chi_{3j}$. Using the decomposition of $Q_j -Q$, Schwarz inequality and the fact 
that $\sqrt{ab}\le 2^{-1}(a+b)$ yields 
\begin{eqnarray}
\left| \chi_{3j}\right| & =& \left|\frac{\tilde d_j}n \E \left( \frac 1{1+\eta_j^* Q_j \eta_j}
\right)\E (u^* RD(Q_j-Q) u)\right| \ ,\nonumber \\
%&\le & \frac Kn \frac{|z|}{\dzz} \E \left| \frac{u^* RD Q_j \eta_j \eta_j^* Q_j u }{1+\eta_j^* Q_j \eta_j} \right|\ ,\nonumber \\
%&\le & \frac Kn \frac{|z|^2}{\dzz^2} \sqrt{ \E | u^* RD Q_j \eta_j|^2 }\sqrt{ \E |\eta_j^* Q_j u |^2}\ ,\nonumber \\
&\le & \frac Kn \frac{|z|^2}{\dzz^2}\left( \E | u^* RD Q_j \eta_j|^2 + \E |\eta_j^* Q_j u |^2\right) \ . \label{eq:chi3}
\end{eqnarray} 
Now, as:
\begin{eqnarray*}
\E | u^* RD Q_j \eta_j|^2 &=& \E u^* RD Q_j y_j y_j^* Q_j^* D R^* u + \E u^* RD Q_j a_j a_j^* Q_j^* D R^* u\ ,  \\
\E |\eta_j^* Q_j u |^2 &=&  \E u^* Q_j^* y_j y_j^* Q_j u + \E u^* Q_j^* a_j a_j^* Q_j u   \ , \\
\end{eqnarray*}
it remains to sum over $j$ and to apply Lemma \ref{lemma:QiaaQi} to
get the estimate $\sum_j |\chi_{3j}| = n^{-1} \Phi(|z|)
\Psi(\dzz^{-1})$. Gathering the partial estimates yields:
\begin{equation}\label{eq:Z1-final}
\bigg|  
\E \sum_j Z_{1j} \bigg| 
\le \frac{\Phi(|z|) \Psi(\dzz^{-1})}{\sqrt{n}}\ .
\end{equation}

\subsection{Convergence to zero of $\sum_j \E Z_{2j}$}
Recall that 
$$
Z_{2j}= \frac{(\alpha \tilde d_j - y_j^* Q_j y_j) u^* R a_j a_j^* Q_j u}
{(1+\eta_j^* Q_j \eta_j)(1+\alpha \tilde d_j)}\ .
$$
We have:
\begin{eqnarray}
| \E Z_{2j}| &\stackrel{(a)}\le & \frac{|z|^2}{\dzz^2} |u^* R a_j| \E\left| 
(\alpha \tilde d_j - y_j^* Q_j y_j) a_j^* Q_j u
\right|\nonumber \\
& \le &
\frac{|z|^2}{\dzz^2} |u^* R a_j| \sqrt{\E |a_j^* Q_j u|^2} \sqrt{\E\left| 
\alpha \tilde d_j - y_j^* Q_j y_j\right|^2}  \nonumber \\
&\le & \frac{|z|^2}{\dzz^2} \left( \frac{
u^* R a_j a_j^* R u + \E u^* Q_j a_j a_j^* Q_j u 
} 2\right)\sqrt{\E\left| 
\alpha \tilde d_j - y_j^* Q_j y_j\right|^2}  \ ,\label{eq:Z2}
\end{eqnarray}
where $(a)$ follows from \eqref{eq:property-ST}. In order to estimate
the remaining square root, we decompose the difference as:
$$
\alpha \tilde d_j - y_j^* Q_j y_j = \frac{\tilde d_j}n \tr D(\E Q -Q) + \frac{\tilde d_j}n  \tr D(Q - Q_j)
+ \frac{\tilde d_j}n \tr D Q_j - y_j^* Q_j y_j\ .
$$
Hence, 
\begin{multline*}
\E |\alpha \tilde d_j - y_j^* Q_j y_j|^2  \\
\le  K\left( 
\frac 1{n^2} \E\left|
\tr D(\E Q - Q)\right|^2 
+ \frac 1{n^2} \E | \tr D(Q -Q_j)|^2  
+ \E\left| \frac{\tilde d_j}n \tr D Q_j - y_j^* Q_j y_j\right|^2 \right)\ .
\end{multline*}

Writing $\E | n^{-1} \tr D(Q -\E Q) |^2 \le \boldsymbol{\ell^+}
\sup_{j\le n} \E |e_j^* D (Q - \E Q) e_j |^2$ where $e_j$ represents
canonical vector number $j$ and using the result of Section
\ref{sec:proof1}, the first term of the right hand side is of order
$n^{-1} \Phi(|z|) \Psi(\dzz^{-1})$. The second term is of order $(n\dzz)^{-2}$
(minor modification of \cite[Lemma 2.6]{SilBai95} to encompass the
case $\Real(z) < 0$). Finally, the third term is of order $n^{-1} \dzz^{-2}$ by
Lemma \ref{lemma:approx-quadra}. Collecting these results, we obtain:
\begin{eqnarray*}
\sqrt{\E |\alpha \tilde d_j - y_j^* Q_j y_j|^2} &\le & \frac K{\sqrt{n} }
\left( 
\Phi_1 \Psi_1 +\frac{\Phi_2 \Psi_2 }n 
+\Phi_3 \Psi_3 
\right)^{1/2} \\
&\le & \frac K{\sqrt{n} } 
\left( 
\Phi_1 \Psi_1 +\Phi_2 \Psi_2
+\Phi_3 \Psi_3
\right)^{1/2}\\
&\stackrel{(a)}\le&  \frac K{\sqrt{n} } \sqrt{\tilde \Phi \tilde \Psi} \quad \stackrel{(b)}\le 
\frac K{\sqrt{n} } \Phi \Psi\ , 
\end{eqnarray*}
where the $\Phi$'s are nice polynomials with argument $|z|$ and the
$\Psi$'s are nice polynomials with argument $|\dzz^{-1}|$, and where
$(a)$ follows from \eqref{eq:rule1} and $(b)$ from
\eqref{eq:rule2}. It remains to plug this estimate into \eqref{eq:Z2},
to sum over $j$ and to use Assumption \ref{ass:A} together with Lemma
\ref{lemma:QiaaQi} to obtain:
\begin{eqnarray}
\bigg| \E \sum_{j=1}^n Z_{2j} \bigg| 
&\le & \frac{K |z|^2}{\sqrt{n} \dzz^2} \left( u^* R A A^* R u + \sum_{j=1}^n\E u^* Q_j a_j a_j^* Q_j u \right)
\Phi(|z|) \Psi (\dzz^{-1})\ , \nonumber \\
&\le & \frac 1{\sqrt{n}} \Phi'(|z|) \Psi' (\dzz^{-1})\ .
\label{eq:Z2-final}
\end{eqnarray}

\subsection{Convergence to zero of $\sum_j \E Z_{3j}$}
Recall that 
$$
Z_{3j}= \frac{y_j^* Q_j u a_j^* Q_j y_j \times u^* R a_j } 
{(1+\eta_j^* Q_j \eta_j)(1+\alpha \tilde d_j)}\ .
$$
We have:
\begin{eqnarray*}
\E |Z_{3j}| &\stackrel{(a)}\le & \frac{|z|^2}{\dzz^2} |u^* R a_j|\times  \E |y_j^* Q_j u a_j^* Q_j y_j| \quad 
 \le \quad \frac{|z|^2}{\dzz^2} |u^* R a_j| \sqrt{\E |y_j^* Q_j u|^2 \E | a_j^* Q_j y_j|^2}\\
&\stackrel{(b)}\le & \frac Kn \frac{|z|^2}{\dzz^4} |u^* R a_j| \ ,
\end{eqnarray*}
where $(a)$ follows from \eqref{eq:property-ST}, and $(b)$ from Lemma \ref{lemma:approx-quadra}. Hence, 
\begin{eqnarray}
\left| \sum_{j=1}^n \E Z_{3j} \right|& \le& \frac Kn \frac{|z|^2}{\dzz^4} \sum_{j=1}^n |u^* R a_j|\nonumber \\
& \le& \frac Kn \frac{|z|^2}{\dzz^4} \sqrt{n}\times \sqrt{ \sum_{j=1}^n u^* R a_j  a_j^* R^* u} \quad
=\quad {\mathcal O}\left( \frac{|z|^2}{\sqrt{n}\, \dzz^5}\right) \ .\label{eq:Z3-final}
\end{eqnarray}
 
\subsection{Convergence to zero of $\sum_j \E Z_{4j}$}
Write $Z_{4j}$ as
$$
Z_{4j}= \frac{ W_{4j}}{(1+\eta_j^* Q_j \eta_j)(1+\alpha \tilde d_j)}
$$
with
\begin{multline*}
W_{4j}  = (1+\alpha \tilde d_j)(u^* R y_j a_j^* Q_j u + u^* R a_j y_j^* Q_j u) \\
- y_j^* Q_j a_j u^* R a_j a_j^* Q_j u 
- a_j^* Q_j y_j u^* R a_j a_j^* Q_j u \\
+ u^* R a_j a_j^* Q_j a_j y_j^* Q_j u 
+ u^* R a_j a_j^* Q_j y_j a_j ^* Q_j u 
\end{multline*}
Write 
$$
\frac 1{1+\eta_j^* Q_j \eta_j} = \frac 1{ 1+\frac{\tilde d_j}n \tr DQ_j + a_j^* Q_j a_j}
- \frac{\Delta_j}{(1+\eta_j^* Q_j \eta_j)(1+\frac{\tilde d_j}n \tr DQ_j + a_j^* Q_j a_j)}\ .
$$
Plugging this identity into $Z_{4j}$ and taking into account the fact
that $\E_{y_j} W_{4j}=0$, we obtain:
\begin{eqnarray*}
\left| \E Z_{4j} \right| &=& \left| 
\E\left(
\frac {\Delta_j W_{4j}}{(1+\alpha \tilde d_j)(1+\eta_j^* Q_j \eta_j)(1+\frac{\tilde d_j}n \tr DQ_j + a_j^* Q_j a_j)} 
\right)
\right|\\
&\le & \frac{|z|^3}{\dzz^3} \sqrt{\E |\Delta_j|^2} \sqrt{\E|W_{4j}|^2}
\quad \le \quad  \frac K{\sqrt{n}} \frac{|z|^3}{\dzz^4} \sqrt{\E|W_{4j}|^2}\ .
\end{eqnarray*}
Hence, 
\begin{equation}\label{eq:Z4}
\left|\E \sum_j Z_{4j}\right| \le \frac K{\sqrt{n}} \frac{|z|^3}{\dzz^4} \sum_j \sqrt{\E|W_{4j}|^2}
\le \frac{K |z|^3}{\dzz^4} \sqrt{\sum_j \E|W_{4j}|^2}\ .
\end{equation}
We therefore estimate $\sum_j \E|W_{4j}|^2$. First, write:
\begin{eqnarray*}
\E|W_{4j}|^2 &\le & \frac Kn \left( 1 + \frac 1\dzz\right)^2 
\left( 
\E |a_j^* Q_j u|^2 u^* R D R^* u + |u^* R a_j| ^2 \E( u^* Q_j^* D Q_j u)  \right)\\
&&\quad + \frac Kn |u^* R a_j|^2 
\E\left[ |a_j^* Q_j u|^2 \left( a_j^* Q_j^* D Q_j a_j + a_j^* Q_j D Q_j^* a_j \right)
\right]\\
&&\qquad  + \frac Kn |u^* R a_j|^2 
\E \left(
|a_j^* Q_j a_j|^2 u^* Q_j^* D Q_j u + |a_j^* Q_j u|^2 a_j^* Q_j D Q_j a_j\ . 
\right)
\end{eqnarray*} 
Now, summing over $j$ yields:
\begin{eqnarray*}
\sum_{j=1}^n \E|W_{4j}|^2 &\le & \frac Kn \left( \sum_{j=1}^n \E ( u^* Q_j^* a_j a_j^* Q_j u)\right)
\left( 1 + \frac 1\dzz\right) \frac 1{\dzz^2}\\
&& +  \frac Kn \left( \sum_{j=1}^n \E ( u^* R a_j a_j^* R^* u)\right)
\left( \frac 1{\dzz^4} + \frac 1{\dzz^2}\left( 1+\frac 1\dzz \right) \right)\\
&\le& \frac 1n \Phi(|z|)\Psi(\dzz^{-1})\ .
\end{eqnarray*}
Plugging this into \eqref{eq:Z4} yields the estimate 
\begin{equation}\label{eq:Z4-final}
\left|\E \sum_j Z_{4j}\right| \le \frac 1{\sqrt{n}} \Phi'(|z|) \Psi'(\dzz^{-1})\ .
\end{equation}

\subsection{End of proof} Recall that:
$$
\left| u^* (R-\E Q) u\right| \quad \le 
\quad 
\bigg| \E \sum_{j=1}^n Z_{1 j} 
\bigg|   + 
\bigg| \E \sum_{j=1}^n Z_{2 j}  
\bigg|   +
\bigg| \E \sum_{j=1}^n Z_{3 j}  
\bigg|   +\bigg| \E \sum_{j=1}^n Z_{4 j}  
\bigg|\ .
$$
It remains to gather estimates \eqref{eq:Z1-final}, \eqref{eq:Z2-final}, \eqref{eq:Z3-final} and \eqref{eq:Z4-final}
to get the desired estimate:
$$
\left| u^* (R-\E Q) u\right|\quad \le \quad \frac 1{\sqrt{n}} \Phi(|z|) \Psi(\dzz^{-1})\ .
$$ 

\section{Proof of Proposition \ref{prop:intermediate3}}\label{sec:proof3}

% supprimé suite à review IHP

Recall the decomposition:
$$
u^*(Q-T)v= u^*(Q-\E Q)v+ u^*(\E Q-R)v+u^*(R-T)v\ .
$$
As mentioned in Section \ref{sec:reduction}, it is sufficient to
establish the estimate:
\begin{equation}\label{eq:R-T}
  \left| u^* \left( R(z) -T(z)\right) u\right| \le
\frac{1}{n} \Phi(|z|)\Psi\left( \frac 1{\dzz}\right)\ , 
\end{equation}
for $z\in \C-\R^+$ in the case where $u$ has norm one. 

\subsection{The estimate for $u^* (R-T)u$} 
Recall
the definitions of $\delta,\tilde \delta$ \eqref{eq:fundamental}, $\alpha,
\tilde \alpha$ \eqref{eq:def-alpha} and $R,\tilde R$
(\ref{eq:def-R}-\ref{eq:def-Rtilde}).
Using twice the resolvent identity yields:
\begin{equation}\label{eq:resolvent-determinist}
u^*(R-T) u = (\tilde \alpha - \tilde \delta)  \kappa_1  + (\alpha - \delta) \kappa_2\ ,
\end{equation}
where
$$
\left\{
\begin{array}{l}
\kappa_1 = z u^* RDT u\\
\kappa_2 = u^* RA(I+\alpha \tilde D) ^{-1} \tilde D (I+ \delta \tilde D)^{-1} A^* T u
\end{array}
\right. \ .
$$
The following bounds are straightforward:
$$
|\kappa_1| \le \frac{|z|   \boldsymbol{\tilde d_{\max}}}{\dzz^2}
\quad \textrm{and}\quad 
|\kappa_2| \le \frac{ \|A\|^2   \boldsymbol{\tilde d_{\max}}}{\dzz^2}\times  \|(I + \alpha\tilde D)^{-1}\| \times 
\|(I + \delta \tilde D)^{-1}\| \ .
$$
It remains to control the spectral norms of $(I + \alpha\tilde
D)^{-1}$ and $(I + \delta \tilde D)^{-1}$.  Recall that $\alpha$ is the
Stieltjes transform of a positive measure with support included in $\R^+$. This
in particular implies that $\Imm(z\alpha)>0$ for $z\in \Cplus$. One can 
check that 
\[
\Upsilon_j(z) = \frac 1{-z(1+\alpha \tilde d_j)}
\] 
is analytic and satisfies $\Imm(\Upsilon_j) > 0$ and $\Imm(z \Upsilon_j) > 0$ 
on $\Cplus$ and that $\lim_{y\rightarrow\infty} (-\ii y \Upsilon_j(\ii y)) 
=1$. As a consequence, $\Upsilon_j$ is the Stieltjes transform of a
probability measure with support included in $\R^+$ 
(see \emph{e.g.} \cite[Prop. 2.2(2)]{HLN07}). In particular,
$$
\left| \Upsilon_j(z)\right|\le \frac 1{\dzz}\quad\textrm{for}\quad 1\le j\le n\ ,
$$
which readily implies that $\|(I + \alpha\tilde D)^{-1}\|\le |z| \dzz^{-1}$.
The same argument applies for $\|(I + \delta \tilde D)^{-1}\|$. Finally,
$$
|\kappa_2|\le \frac{|z|^2  \|A\|^2   \boldsymbol{\tilde d_{\max}}}{\dzz^4}\ .
$$
In view of the estimates obtained for $\kappa_1$ and $\kappa_2$, it is
sufficient, in order to establish \eqref{eq:R-T}, to obtain estimates
for $\alpha - \delta$ and $\tilde \alpha - \tilde \delta$. Assume that
the following estimate holds true:
\begin{equation}\label{eq:estimate-alpha-m}
  \forall z\in \C- \R^+\ ,\quad 
  \max\left( |\alpha - \delta|, |\tilde \alpha - \tilde \delta| \right) 
\le \frac 1n \Phi(|z|) \Psi\left( \frac 1{\dzz}\right)\ , 
\end{equation}
where $\Phi$ and $\Psi$ are nice polynomials. Then, plugging
\eqref{eq:estimate-alpha-m} into \eqref{eq:resolvent-determinist}
immediately yields the desired result \eqref{eq:R-T}.

The rest of the section is devoted to establish \eqref{eq:estimate-alpha-m}.

\subsection{Auxiliary estimates over $(\alpha- \delta)$ and $(\tilde \alpha - \tilde
  \delta)$}

Writing $\alpha = n^{-1} \tr DR +n^{-1} \tr D(\E Q -R)$ and
$\delta=n^{-1} \tr DT$, the difference $\alpha-\delta$ expresses as
$n^{-1} \tr D(R-T) + n^{-1} \tr D(\E Q -R)$. Now using the resolvent
identity $R-T=-R(R^{-1} - T^{-1})T$ and performing the same
computation for the tilded quantities yields the following system of
equations:
\begin{equation}\label{eq:system-C0}
\left( 
\begin{array}{c}
\alpha-\delta\\
\tilde \alpha - \tilde \delta
\end{array}
\right) = C_0 
\left( 
\begin{array}{c}
\alpha-\delta\\
\tilde \alpha - \tilde \delta
\end{array}
\right) + 
\left( 
\begin{array}{c}
\boldsymbol{\varepsilon}\\
\boldsymbol{\tilde \varepsilon}
\end{array}
\right)\quad \textrm{where}\quad 
C_0= 
\left( 
\begin{array}{cc}
u_0 & zv_0\\
z\tilde v_0 & \tilde u_0
\end{array}
\right)\ ,
\end{equation}
the coefficients being defined as:
\begin{equation}\label{eq:definitions}
\left\{
\begin{array}{ccl}
u_0 &=& \frac 1n \tr D^{1/2} RA (I+\alpha \tilde D)^{-1} \tilde D (I+\delta \tilde D)^{-1} A^* T D^{1/2}\\
\tilde u_0 &=& \frac 1n \tr \tilde D^{1/2} \tilde R A^* (I+\tilde \alpha D)^{-1} D 
(I+\tilde \delta D)^{-1} A \tilde T \tilde D^{1/2}
\\
v_0&=& \frac 1n \tr DRDT \\
\tilde v_0&=& \frac 1n \tr \tilde D \tilde R\tilde D \tilde T\\
\end{array}
\right.\ ,
\end{equation}
and the quantities $\boldsymbol{\varepsilon}$ and $\boldsymbol{\tilde \varepsilon}$ being given by:
\begin{equation}\label{eq:epsilons}
\boldsymbol{\varepsilon}\ =\ \frac 1n \tr D(\E Q -R)\quad \textrm{and} \quad 
\boldsymbol{\tilde \varepsilon}\ =\ \frac 1n \tr \tilde D(\E \tilde Q -\tilde R)\ .
\end{equation}

The general idea, in order to transfer the estimates over
$\boldsymbol{\varepsilon}$ and $\boldsymbol{\tilde \varepsilon}$ (as
provided in Proposition \ref{prop:intermediate2}-(ii) ), to $\alpha-\delta$ and
$\tilde \alpha -\tilde \delta$, is to obtain an estimate over
$1/\det(I-C_0)$, and then to solve the system \eqref{eq:system-C0}.

\subsubsection*{Lower bound for $\det(I-C_0)$}  The mere definition of $I-C_0$ yields 
\begin{eqnarray*}
|\det(I-C_0)| &=& \left| (1-u_0)(1-\tilde u_0) -z^2 v_0 \tilde v_0 \right|\\
&\ge& (1-|u_0|)\times (1-|\tilde u_0|) -|z|^2 |v_0|\times  |\tilde v_0| 
\end{eqnarray*}
In order to control the quantities $u_0,\tilde u_0, v_0$ and $\tilde
v_0$, we shall use the following inequality:
\begin{equation}\label{eq:schwarz}
\left| \tr A B^* \right| \le \left( \tr A A^* \right)^{1/2} \times
\left( \tr B B^* \right)^{1/2}\ ,
\end{equation}
together with the following quantities:
\[
\left\{
\begin{array}{ccl}
u_1 &=& \frac 1n \tr D TA (I+\delta^* \tilde D)^{-1} \tilde D (I+\delta \tilde D)^{-1} A^* T^* \\
\tilde u_1 &=& \frac 1n \tr \tilde D \tilde T A^* (I+\tilde \delta D)^{-1} D 
(I+\tilde \delta^* D)^{-1} A \tilde T^* 
\\
v_1&=& \frac 1n \tr DTDT^* \\
\tilde v_1&=& \frac 1n \tr \tilde D \tilde T\tilde D \tilde T^*\\
\end{array}
\right.
\] 
and
\begin{equation}
\label{eq-C2} 
\left\{
\begin{array}{ccl}
u_2 &=& \frac 1n \tr D RA (I+\alpha^* \tilde D)^{-1} \tilde D (I+\alpha \tilde D)^{-1} A^* R^* \\
\tilde{u}_2 &=& \frac 1n \tr \tilde D \tilde R A^* (I+\tilde \alpha D)^{-1} D 
(I+\tilde \alpha^* D)^{-1} A \tilde R^* 
\\
v_2&=& \frac 1n \tr DRDR^* \\
\tilde{v}_2&=& \frac 1n \tr \tilde D \tilde R\tilde D \tilde R^*\\
\end{array}
\right.
\end{equation} 
Using \eqref{eq:schwarz} together with identity $(I+\delta \tilde D)^{-1}A^*
T = \tilde T A^* (I+\tilde \delta D)^{-1}$ (and similar ones for related
quantities), we obtain:
$$
|u_0| \le (\tilde u_1 u_2)^{1/2}\ ,\quad 
|\tilde u_0| \le (u_1 \tilde{u}_2)^{1/2}\ ,\quad 
|v_0| \le (v_1 v_2 )^{1/2}\ ,\quad 
|\tilde v_0| \le (\tilde v_1 \tilde{v}_2 )^{1/2}\ ,
$$
hence the lower bound:
\begin{equation}\label{eq:lower-determinant}
|\det ( I -C_0)|\ge (1- (\tilde u_1 u_2 )^{1/2}) (1- (u_1 \tilde{u}_2)^{1/2})
- |z|^2 ( v_1 v_2 \tilde v_1 \tilde v_2 )^{1/2}\ .
\end{equation}
Notice that it is not proved yet that the right hand side of the
previous inequality is nonnegative. 

In order to handle estimate \eqref{eq:lower-determinant}, we shall
rely on the following proposition.

\begin{prop}\label{prop:lower-estimate-det}
Consider the nonnegative real numbers $x_i,y_i,s_i,t_i$ ($i=1,2$). Assume that: 
$$
x_i \leq 1,\ y_i \leq 1\qquad 
\textrm{and}\qquad (1-x_i)(1-y_i) -s_i t_i \ge 0\quad \textrm{for}\quad i=1,2.
$$
Then:
\begin{multline*}
(1-\sqrt{x_1 x_2})\left( 1-\sqrt{y_1 y_2}\right) -\sqrt{s_1 s_2 t_1 t_2}\\
\ge \sqrt{ (1-x_1)(1-y_1) -s_1 t_1} \sqrt{ (1-x_2)(1-y_2) -s_2 t_2}\ .  
\end{multline*}
\end{prop}
\begin{proof} If $a \ge c\ (\, \ge 0)$ and $b \ge d\ (\, \ge 0)$, then:
$$
\sqrt{ab} - \sqrt{cd} \ge \sqrt{a-c}\sqrt{b-d}\ .
$$
To prove this, simply take the difference of the squares. Applying once this inequality yields
$1- \sqrt{x_1 x_2} \ge \sqrt{(1-x_1)(1-x_2)} $, hence: 
\begin{multline*}
(1-\sqrt{x_1 x_2})\left( 1-\sqrt{y_1 y_2}\right) -\sqrt{s_1 s_2 t_1 t_2}
\ge \sqrt{(1-x_1)(1-x_2)(1-y_1)(1-y_2)} - \sqrt{s_1 s_2 t_1 t_2}
\end{multline*}
Applying again the first inequality yields then the desired result.
\end{proof}

Our goal is to apply Proposition \ref{prop:lower-estimate-det} to
\eqref{eq:lower-determinant}.  The main idea, in order to fulfill
assumptions of Proposition \ref{prop:lower-estimate-det} (at least on
some portions of $\C-\R^+$), is to consider the quantities of interest,
i.e. $u_i, \tilde u_i, v_i, \tilde v_i$ ($i=1,2$) as coefficients of
linear systems whose determinants are the desired quantities
$(1-u_i)(1-\tilde u_i) -|z|^2 v_i \tilde v_i$.

Consider the following matrices:
$$
C_i(z)=
\left(
\begin{array}{cc}
u_i & v_i\\
|z|^2 \tilde v_i & \tilde u_i 
\end{array}
\right),\quad i=1,2\ .
$$
The following proposition holds true:
\begin{prop}\label{prop:auxiliary-det} 
Assume that $z \in \C - \R^+$. Then: 
\begin{itemize}
\item[(i)] The following holds true: $1 - u_1(z) \ge 0$ and $1 - \tilde u_1(z)
\ge 0$. Moreover, there exists positive constants $K, \eta$ such that:
$$
\det(I-C_1(z))\ \ge\ K \frac{ \dzz^8}{(\eta^2 + |z|^2)^4} \ . 
$$ 
\item[(ii)] There exist nice polynomials $\Phi$ and $\Psi$ and a set 
$$
{\mathcal E}_n=\left\{z\in \C^+, \quad
\frac 1n   \Phi(|z|)\Psi\left(\frac 1{\dzz}\right)\le 1/2\right\}\ ,
$$ 
such that for every $z\in {\mathcal E}_n$, $1-u_2(z) \ge 0$, 
$1 - \tilde u_2(z) \ge 0$, and 
$$
\det(I-C_2) \ge K \frac{ \dzz^8}{(\eta^2 + |z|^2)^4}\ ,
$$
where $K,\eta$ are positive constants. 
\end{itemize}
\end{prop}
\noindent Proof of Proposition \ref{prop:auxiliary-det} is postponed to 
Appendix \ref{app:sec:proof3}. 

We are now in position to establish the following estimate:
\begin{equation}\label{eq:estimate-alpha-m-partial}
  \forall z\in {\mathcal E}_n,\quad   
  \max\left( |\alpha - \delta|, |\tilde \alpha - \tilde \delta| \right) 
\le \frac 1n \Phi(|z|) \Psi\left( \frac 1{\dzz}\right)\ .
\end{equation} 
Assume $z\in {\mathcal E}_n$. Thanks to Proposition \ref{prop:auxiliary-det},
assumptions of Proposition \ref{prop:lower-estimate-det} are fulfilled by 
$u_i,\tilde u_i, v_i$ and $\tilde v_i$, and \eqref{eq:lower-determinant} 
yields:
\begin{equation}\label{eq:lower-estimate-C0}
\det(I-C_0)\ \ge\  \sqrt{\det(I-C_1)}\sqrt{\det(I-C_2)} \ \ge\ 
K\frac{\dzz^8}{(\eta^2 + |z|^2)^4}\ ,
\end{equation}
where $K, \eta$ are nice constants. \\ 
Solving now the system \eqref{eq:system-C0}, we obtain:
$$
\left\{
\begin{array}{lcl}
\alpha - \delta &=& \left( \det(I-C_0)\right)^{-1} \left( (1-\tilde u_0) 
\boldsymbol{\varepsilon} +z v_0 \boldsymbol{\tilde \varepsilon} \right)\\
\tilde \alpha - \tilde \delta &=& \left( \det(I-C_0)\right)^{-1} 
 \left( (1- u_0) 
\boldsymbol{\tilde \varepsilon} +z \tilde v_0 \boldsymbol{\varepsilon} \right)
\end{array}\right.
$$
It remains to use \eqref{eq:lower-estimate-C0}, Proposition 
\ref{prop:intermediate2}-(ii), and obvious bounds over $u_0,\tilde
u_0,v_0$ and $\tilde v_0$ to conclude and obtain
\eqref{eq:estimate-alpha-m-partial}.

We turn out to the case where $z\in \C - \R^+ - {\mathcal E}_n$,
and rely on the same argument as in Haagerup and Thorbjornsen
\cite{HaaTho05} (see also \cite{CDF09}). In this case,
$$ 
\frac 1n \Phi(|z|)\Psi(\dzz^{-1}) \ge \frac 12 \ .
$$
As $|\alpha - \delta|=|n^{-1}
\tr D(\E Q - T)|\le 2 \boldsymbol{\ell^+} \boldsymbol{d_{\max}} \dzz^{-1}$, 
we obtain:
$$
\forall z \in \C - \R^+ - {\mathcal E}_n,  
\quad |\alpha - \delta| \le \frac
{2 \boldsymbol{\ell^+} \boldsymbol{d_{\max}}}{\dzz} 
\times \frac{2\Phi(|z|) \Psi\left( \frac 1{\dzz}\right)}n \ ;
$$
a similar estimate holds for $\tilde \alpha - \tilde \delta$ for
$z\notin {\mathcal E}_n$.  Gathering the cases where $z\in {\mathcal
  E}_n$ and $z\notin {\mathcal E}_n$ yields \eqref{eq:estimate-alpha-m}. 

\begin{appendix}
\section{Remaining proofs for Section \ref{sec:misc}}
\label{app:proof-misc}

\subsection*{Proof of Lemma \ref{lemma:QaaQ}}
Note that it is sufficient to establish the result for a vector $u$
with norm one (which is assumed in the sequel). The general result
follows by considering $u/\|u\|$.
 
We proceed by induction over $p$. Let $p=1$ and consider: 
$$
0\ \le\ \E \sum_{j=1}^n \E_{j-1} u^* Q a_j a_j^* Q^* u\ =\ \E u^* Q A
A^* Q^* u\ \le\ \boldsymbol{a_{\max}}^2 \E \|Q\|^2\ .
$$
As $\|Q\| \le \dzz^{-1}$, we obtain the desired bound.

Now, write
\begin{eqnarray*}
\E \left|\sum_{j=1}^n \E_{j-1} (u^* Q a_j a_j^* Q^* u )\right|^p &=&
\sum_{j_1,\cdots, j_p}  \E \left[ 
\E_{j_1-1} (u^* Q a_{j_1} a_{j_1}^* Q^* u )\cdots \E_{j_p-1} (u^* Q a_{j_p} a_{j_p}^* Q^* u )
\right]\\
&\le&
p! \sum_{j_1\le \cdots\le j_p}  \E \left[ 
\E_{j_1-1} (u^* Q a_{j_1} a_{j_1}^* Q^* u )\cdots \E_{j_p-1} (u^* Q a_{j_p} a_{j_p}^* Q^* u )
\right]\\
&=&
p! \sum_{j_1\le \cdots\le j_p}  \E \Big[
\E_{j_p-1} (u^* Q a_{j_p} a_{j_p}^* Q^* u ) 
\underbrace{\prod_{k=1}^{p-1} \E_{j_k-1} (u^* Q a_{j_k} a_{j_k}^* Q^* u ) }_{{\mathcal F}_{j_p-1}\ \textrm{measurable}}
\Big]\\
&=&
p! \sum_{j_1\le \cdots\le j_{p-1}}  \E \Big[
\sum_{j_p=j_{p-1}}^n (u^* Q a_{j_p} a_{j_p}^* Q^* u ) 
\prod_{k=1}^{p-1} \E_{j_k-1} (u^* Q a_{j_k} a_{j_k}^* Q^* u ) 
\Big]\\
&\stackrel{(a)}{\le}&
p!\ \frac{\boldsymbol{a_{\max}}^2 }{\dzz^{2}}  
\ \E \left|\sum_{j=1}^n \E_{j-1} (u^* Q a_j a_j^* Q^* u )\right|^{p-1}\ ,
\end{eqnarray*}
where $(a)$ follows from the fact that $$
\sum_{j_p=j_{p-1}}^n (u^* Q a_{j_p} a_{j_p}^* Q^* u ) \le 
\sum_{j_p=1}^n (u^* Q a_{j_p} a_{j_p}^* Q^* u )\le 
\frac{\boldsymbol{a_{\max}}^2}{\dzz^2}\ .
$$  It
remains to plug the induction assumption to conclude. Hence
\eqref{eq:QaaQ} is established. 

In order to establish \eqref{eq:QetaetaQ}, one may use the same
arguments as previously together with the identity $Q\Sigma\Sigma^* =
I +zQ$, which yields the factor $|z|^p$ in estimate \eqref{eq:QetaetaQ}.

\subsection*{Proof of Lemma \ref{lemma:QiaaQi}} We prove the lemma in
the case where $\|u\|=1$, the general result readily follows by
considering $u/\|u\|$.

Write $u^* Q_j a_j a_j^* Q_j ^* u= \chi_{1j} + 
\chi_{2j} + \chi_{3j} + \chi_{4j}$ with:
\begin{eqnarray*}
\chi_{1j}&=& u^* (Q_j-Q) a_j a_j^* (Q_j-Q) ^* u \\
\chi_{2j}&=& u^* Q a_j a_j^* Q ^* u\\
\chi_{3j}&=& u^* (Q_j-Q) a_j a_j^* Q ^* u \\
\chi_{4j}&=&u^* Q a_j a_j^* (Q_j-Q) ^* u 
\end{eqnarray*}
Hence, 
$$
 \sum_{j=1}^n \E \left( u^* Q_j a_j a_j^* Q_j^* u\right)^2 
\le  \sum_{j=1}^n \E  \chi_{1j}^2  +  \sum_{j=1}^n \E \chi_{2j}^2  +  \sum_{j=1}^n \E |\chi_{3j}|^2  
+  \sum_{j=1}^n \E  |\chi_{4j}|^2\ . 
$$
Notice that:
$$
\E |\chi_{3j}|^2 \le \frac 12 \left( 
\E \chi_{1j}^2+ \E \chi_{2j}^2
\right)\quad \textrm{and}\quad 
\E |\chi_{4j}|^2 \le \frac 12 \left( 
\E \chi_{1j}^2+ \E \chi_{2j}^2
\right)\ .
$$
Note that using the facts that $a_j a_j^* \le A A^*$ and $\eta_j \eta_j^* \le \Sigma \Sigma^*$ 
together with the identity $Q \Sigma \Sigma^* = I +zQ$ yield the rough but useful estimates:
\begin{equation}\label{eq:rough-estimate}
u^* Q a_j a_j^* Q^* u ={\mathcal O}\left( 
\dzz^{-2}
\right)\quad \textrm{and}\quad u^* Q \eta_j \eta_j^* Q^* u ={\mathcal O}\left( 
\frac{|z|}{\dzz^2}
\right)\ .
\end{equation}
We first begin by the contribution of $\sum_j \E \chi_{2j}^2$:
\begin{eqnarray}
\sum_{j=1}^n \chi_{2j}^2 &=& \sum_{j=1}^n  u^* Q a_j a_j^* Q ^* u \times u^* Q a_j a_j^* Q ^* u\ ,\nonumber \\
&\le & \sum_{j=1}^n u^* Q a_j a_j^* Q ^* u \times u^* Q A A^* Q ^* u\ ,\nonumber \\
&\le & \left( u^* Q A A^* Q ^* u\right)^2 \quad =\quad {\mathcal O}\left( \dzz^{-4}
\right)\ \nonumber \\
&\le & \Phi_2(|z|)\Psi_2\left( \frac 1\dzz\right)\ .\label{eq:phi2}
\end{eqnarray}
Similarly,
\begin{equation}\label{app:estimate1}
\sum_{j=1}^n \left( u^* Q \eta_j \eta_j^* Q^* u \right)^2 \quad =\quad {\mathcal O}\left( 
\frac{|z|^2}{\dzz^{4}}
\right)\ .
\end{equation}
We now turn to the contribution of $\sum_j \E \chi_{1j}^2$. Using the decompositions \eqref{eq:woodbury}
, \eqref{eq:woodbury2} and \eqref{eq:woodbury3}, $\chi_{1j}$ writes:
\begin{eqnarray}
\chi_{1j}&=& \left| 
\frac{1+\eta_j^* Q_j \eta_j}{1-\eta_j^* Q \eta_j} 
\right| \times  \left| u^* Q \eta_j \eta_j^* Q a_j a_j^* Q^* \eta_j \eta_j^* Q^* u\right| \nonumber\\
&=& \left| 
1+\eta_j^* Q_j \eta_j \right| \times \left| u^* Q \eta_j \eta_j^* Q^* u \right|\times  \left| 
\frac{ a_j^* Q^* \eta_j \eta_j^* Q a_j}{1-\eta_j^* Q \eta_j}\right| \ .\label{eq:decomposition-chi1}
\end{eqnarray}
We first prove that 
\begin{equation}\label{eq:cross-estimate}
\frac{ a_j^* Q^* \eta_j \eta_j^* Q a_j}{1-\eta_j^* Q \eta_j}={\mathcal O}\left(\frac{|z|}
{\dzz^2}\right)\ .
\end{equation}
In fact:
\begin{eqnarray*}
\left| \frac{ a_j^* Q^* \eta_j \eta_j^* Q a_j}{1-\eta_j^* Q \eta_j} \right| &\le&
\left| \frac{ a_j^* Q^* \eta_j \eta_j^* Q^* a_j}{1-\eta_j^* Q \eta_j}\right|
+ \left| \frac{ a_j^* Q^* \eta_j \eta_j^* (Q-Q^*) a_j}{1-\eta_j^* Q \eta_j}\right|\\
&\stackrel{(a)}\le& \left| a_j^* (Q_j -Q)^*a_j  \right| + 2|\mathrm{Im}(z)| |a_j^* (Q_j -Q)Q a_j|\\
&=& {\mathcal O}\left(\frac 1{\dzz}\right) 
+ {\mathcal O}\left(\frac{|z|}{\dzz^2}\right)\quad = \quad 
{\mathcal O}\left(\frac{|z|}{\dzz^2}\right)\ ,
\end{eqnarray*}
where we use the fact that $Q - Q^*=2{\bf i}\mathrm{Im}(z) Q^* Q$ to
obtain $(a)$. Now,
\begin{eqnarray} \label{eq:triangular-estimate}
\left| 
1+\eta_j^* Q_j \eta_j \right|\le 1 + \left| \Delta_j
\right| 
+\left|\frac {\tilde d_j}n \tr D Q_j + a_j^* Q_j a_j
\right|\ .
\end{eqnarray}
Since $|n^{-1}\tilde d_j \tr D Q_j + a_j^* Q_j a_j| ={\mathcal O}(\dzz^{-1})$, we obtain:
\begin{eqnarray*}
  \sum_{j=1}^n \E \chi_{1j}^2& = & 
\left(   {\mathcal O}\left( \frac{|z|^2}{\dzz^4}\right)
+ {\mathcal O}\left( \frac{|z|^2}{\dzz^6}\right)\right)
\times \sum_{j=1}^n \E \left( u^* Q \eta_j \eta_j^* Q^* u \right)^2\\
&&
+ {\mathcal O}\left( \frac{|z|^2}{\dzz^4}\right)
\times \sum_{j=1}^n \E \left( u^* Q \eta_j \eta_j^* Q^* u \right)^2\times \left| 
\Delta_j
\right|^2\\
&\stackrel{(a)}{=}& {\mathcal O}\left( \frac{|z|^4}{\dzz^8}\right)
+ {\mathcal O}\left( \frac{|z|^4}{\dzz^{10}}\right) + 
{\mathcal O}\left( \frac{|z|^4}{\dzz^8}\right)
\times \sum_{j=1}^n \E  \left| \Delta_j
\right|^2\\
&\stackrel{(b)}{=}& {\mathcal O}\left( \frac{|z|^4}{\dzz^8}\right)
+ {\mathcal O}\left( \frac{|z|^4}{\dzz^{10}}\right) \\
&\le& \Phi_1(|z|) \Psi_1\left( \frac 1{\dzz}\right)\ ,
\end{eqnarray*}
where $(a)$ follows from \eqref{app:estimate1} and
\eqref{eq:rough-estimate} and $(b)$, from Corollary \ref{coro:quadra}.

It remains to gather the contributions of $\chi_{1j}, \chi_{2j},
\chi_{3j}$ and $\chi_{4j}$ to get:
$$
\sum_{j=1}^n \E \left( u^* Q_j a_j a_j^* Q_j^* u\right)^2 
\quad \le \quad  
2\Phi_1(|z|) \Psi_1\left( \frac 1\dzz \right) + 2\Phi_2(|z|) \Psi_2\left( \frac 1\dzz \right) 
\quad \stackrel{(a)}\le \quad  \Phi(|z|)\ \Psi\left( \frac 1\dzz \right)\ ,
$$
where $(a)$ follows from \eqref{eq:rule1}.
Eq. \eqref{eq:QiaaQi1} is proved.

In order to prove \eqref{eq:QiaaQi2}, first note that:
\begin{multline*}
\E \bigg( 
\sum_{j=1}^n \E_{j-1} \left(u^* Q_j a_j a_j^* Q_j^* u\right)\ \bigg)^p \\
\le K \left( \E \bigg| \sum_{j=1}^n \E_{j-1} \chi_{1j} \bigg|^p
+ \E \bigg| \sum_{j=1}^n \E_{j-1} \chi_{2j} \bigg|^p
+\E \bigg| \sum_{j=1}^n \E_{j-1} \chi_{3j} \bigg|^p
+\E \bigg| \sum_{j=1}^n \E_{j-1} \chi_{4j} \bigg|^p\right)\ .
\end{multline*}
Hence, it remains to evaluate the contributions of each term.  Using
decomposition \eqref{eq:decomposition-chi1} together with the estimate
\eqref{eq:cross-estimate}, we obtain:
$$
\E \bigg| \sum_{j=1}^n \E_{j-1} \chi_{1j} \bigg|^p = {\mathcal
  O}\left( \frac{|z|^p }{\dzz^{2p}}\right) 
\times \E \bigg( 
\sum_{j=1}^n \E_{j-1} |1+\eta_j^* Q_j \eta_j| \times u^* Q\eta_j
\eta_j^* Q^* u \bigg)^p \ .
$$
Using \eqref{eq:triangular-estimate} together with \eqref{eq:QetaetaQ} yields:
\begin{eqnarray*}
\E \bigg| \sum_{j=1}^n \E_{j-1} \chi_{1j} \bigg|^p &=& {\mathcal
  O}\left( \frac{|z|^{2p} }{\dzz^{4p}}\right) +  {\mathcal
  O}\left( \frac{|z|^{2p} }{\dzz^{5p}}\right) + {\mathcal
  O}\left( \frac{|z|^p }{\dzz^{2p}}\right)  \times \E \bigg| \sum_{j=1}^n \E_{j-1} \left( 
|\Delta_j| \times u^* Q\eta_j
\eta_j^* Q^* u\right) \bigg|^p \ .
\end{eqnarray*}
Combining standard inequalities (Cauchy-Schwarz, $|\sum_j a_j b_j |
\le (\sum_j a_j^2)^{1/2} (\sum_j b_j^2)^{1/2}$, and Cauchy-Schwarz
again), we obtain:
\begin{eqnarray*}
\lefteqn{ \E \bigg( \sum_{j=1}^n \E_{j-1} \left(  |\Delta_j
| \times u^* Q\eta_j \eta_j^* Q^* u \right) \bigg)^p}\\
&\le& \left[ 
\E \bigg( \sum_{j=1}^n \E_{j-1} ( u^* Q\eta_j
\eta_j^* Q^* u)^2 \bigg)^p 
\times 
\E \bigg( \sum_{j=1}^n \E_{j-1} |\Delta_j|^2  \bigg)^p \right]^{1/2} 
\ \stackrel{(a)}=\  {\mathcal O}\left( \frac{|z|^{p}}{\dzz^{3p}}
\right) \ ,
\end{eqnarray*}
where $(a)$ follows from \eqref{eq:rough-estimate}, Corollary \ref{coro:quadra} and \eqref{eq:QetaetaQ}.
Finally,
\begin{equation}\label{eq:contrib-chi1}
\E \bigg| \sum_{j=1}^n \E_{j-1} \chi_{1j} \bigg|^p\  =\ {\mathcal
  O}\left( \frac{|z|^{2p} }{\dzz^{4p}}\right) +  {\mathcal
  O}\left( \frac{|z|^{2p} }{\dzz^{5p}}\right) + {\mathcal O}\left( \frac{|z|^{2p}}{\dzz^{5p}}
\right) 
\ \le\  \Phi_1(|z|) \Psi_1(\dzz^{-1})\ .
\end{equation}
Eq. \eqref{eq:QaaQ} directly yields the estimate: 
\begin{equation}\label{eq:contrib-chi2}
\E \bigg| \sum_{j=1}^n \E_{j-1} \chi_{2j} \bigg|^p = {\mathcal
  O}\left( \frac 1{\dzz^{2p}}\right) \le \Phi_2(|z|) \Psi_2(\dzz^{-1}) \ . 
\end{equation}
Finally, 
\begin{equation}
\E \bigg| \sum_{j=1}^n \E_{j-1} \chi_{3j} \bigg|^p\ \le\ 
\left( \E \bigg| \sum_{j=1}^n \E_{j-1} \chi_{1j} \bigg|^p 
\E \bigg| \sum_{j=1}^n \E_{j-1} \chi_{2j} \bigg|^p \right)^{1/2}\
\le\ \Phi_3(|z|) \Psi_3\left(\frac 1{\dzz}\right)\ .\label{eq:contrib-chi3}
\end{equation}
A corresponding inequality exists 
for $\E | \sum \E_{j-1} \chi_{4j} |^p$:
obtain:
\begin{equation}\label{eq:contrib-chi4}
\E \bigg| \sum_{j=1}^n \E_{j-1} \chi_{4j} \bigg|^p \ \le \ \Phi_4(|z|) \Psi_4\left(\frac 1{\dzz}\right)\ .
\end{equation}
Gathering \eqref{eq:contrib-chi1}, \eqref{eq:contrib-chi2},
\eqref{eq:contrib-chi3} and \eqref{eq:contrib-chi4}, we end up with
\eqref{eq:QiaaQi2}, and Lemma \ref{lemma:QiaaQi} is proved.

\section{Remaining proofs for Section \ref{sec:proof3}}\label{app:sec:proof3}

\subsection*{Proof of Proposition \ref{prop:auxiliary-det}-(i)}
Recall that $\delta=\frac 1n \tr DT$ and $\tilde \delta = 
\frac 1n \tr \tilde D \tilde T$. We consider first the case where
$z \in \C^+ \cup \C^-$. We have 
$$
\mathrm{Im}(\delta) =  
\frac 1{2\mathbf{i} n} \tr DT( T^{-*} -T^{-1})T^*\quad \textrm{and} \quad 
\mathrm{Im}(z \tilde \delta) = \frac 1{2\mathbf{i} n} \tr \tilde D (z\tilde T) 
\left[ (z\tilde T)^{-*}
- (z\tilde T)^{-1}\right] (z \tilde T)^*\ .
$$
Developing the previous identities, we end up with the system:
\begin{equation}\label{eq:system-aux-1}
\left( I - C_1 \right) 
\left(
\begin{array}{c}
\Imm(\delta)\\
\Imm(z\tilde \delta)
\end{array}
\right) = 
\Imm(z)\ \left(
\begin{array}{c}
w_1(z)\\
\tilde x_1(z)
\end{array}
\right)
\end{equation}
where
$$
\left\{
\begin{array}{rcll}
w_1(z) & = & \frac 1n \tr DT T^* & (>0)\\
\tilde x_1(z) &=& \frac 1n \tr \tilde D \tilde T A^* (I+\tilde \delta D)^{-1} (I + \tilde \delta^* D)^{-1} 
A \tilde T^*&(>0) 
\end{array}\right.\ .
$$
By developing the first equation of this system, and by recalling that
$\delta(z)$ is the Stieltjes transform of a positive measure $\mu_n$ with
support included in $\R^+$, we obtain 
\[
1 - u_1 = w_1  \frac{\Imm(z)}{\Imm(\delta)} + 
v_1  \frac{\Imm(z\tilde\delta)}{\Imm(\delta)} 
\geq w_1  \frac{\Imm(z)}{\Imm(\delta)} \geq 0 \ . 
\] 
Replacing $(\Imm(\delta), \Imm(z\tilde\delta))$ with 
$(\Imm(\tilde\delta), \Imm(z\delta))$ and repeating the same argument, we 
obtain 
\[
1 - \tilde u_1 = \tilde w_1  \frac{\Imm(z)}{\Imm(\tilde \delta)} + 
\tilde v_1  \frac{\Imm(z\delta)}{\Imm(\tilde \delta)} 
\geq \tilde w_1  \frac{\Imm(z)}{\Imm(\tilde \delta)} \geq 0 \ . 
\]
By continuity of $u_1(z)$ and $\tilde u_1(z)$ at any point of the open 
real negative axis, we have $1-u_1 \geq 0$ and $1-\tilde u_1 \geq 0$ 
for any $z\in \C - \R^+$. The first two inequalities in the 
statement of Proposition \ref{prop:auxiliary-det}-(i) are proven. \\
By applying Cramer's rule (\cite[Sec. 0.8.3]{HorJoh94}) where the first
column of $I-C_1$ is replaced with the right hand member of 
\eqref{eq:system-aux-1}, we obtain 
\begin{equation} 
\det(I-C_1) = (1 - \tilde u_1) w_1 \frac{\Imm(z)}{\Imm(\delta)} +
v_1 \tilde x_1 \frac{\Imm(z)}{\Imm(\delta)}  
\geq 
(1 - \tilde u_1) w_1 \frac{\Imm(z)}{\Imm(\delta)}  
\geq 
w_1\tilde w_1  \frac{\Imm(z)}{\Imm(\delta)}\frac{\Imm(z)}{\Imm(\tilde \delta)}.
\label{eq-det(I-C1)} 
\end{equation} 
Using the fact that the positive measure $\mu_n$ is supported by $\R^+$ and has
a total mass $n^{-1}\tr D$, we have 
\begin{equation}
\label{eq-bound-Imz/delta} 
0\leq\frac{\Imm(\delta)}{\Imm(z)} = \int \frac{1}{|t-z|^2} \mu_n(dt) 
\leq \frac{1}{\bs\delta_z^2} \frac 1n \tr D \leq
\frac{\boldsymbol{\ell^+} \boldsymbol{d_{\max}}}{\bs\delta_z^2}  , 
\quad \text{and} \quad 
0\leq\frac{\Imm(\tilde\delta)}{\Imm(z)}  
\leq 
\frac{\boldsymbol{\tilde d_{\max}}}{\bs\delta_z^2} \ . 
\end{equation} 
In order to find a lower bound on $w_1$ and $\tilde w_1$, we begin by finding
a lower bound on $|\delta|$. \\
A computation similar to \cite[Lemma C.1]{HLN07} shows that the sequence
of measures $(\mu_n)$ is tight. Hence there exists $\eta>0$ such that:
\[
\mu_n[0,\eta]\ge \frac 12\  
\frac 1n \tr D \ge \frac{\boldsymbol{\ell^-}\boldsymbol{d_{\min}}}2\ . 
\]
We have 
\begin{equation}
\label{eq-inf-delta} 
|\delta| \ge |\Imm(\delta)| \ =\  
|\Imm(z)| \int \frac{\mu_n(dt)}{|t - z|^2}
\ \ge \ |\Imm(z)| \int_0^\eta \frac{\mu_n(dt)}{2(t^2 + |z|^2)}\\
\ge |\Imm(z)| \frac{ \boldsymbol{\ell^-}\boldsymbol{d_{\min}} } 
{4(\eta^2 + |z|^2)} \ . 
\end{equation} 
Furthermore, when $\Real(z) < 0$, we have 
\[
|\delta| \ge \Real(\delta)\ =\  
\int \frac{t - \Real(z)}{|t - z|^2} \mu_n(dt) \ge 
- \Real(z) \int \frac{\mu_n(dt)}{|t - z|^2} 
\ge - \Real(z)  \frac{ \boldsymbol{\ell^-}\boldsymbol{d_{\min}} } 
{4(\eta^2 + |z|^2)} \ . 
\]
which results in 
\[
|\delta| \ge {\bs\delta}_z \frac{ \boldsymbol{\ell^-}\boldsymbol{d_{\min}} } 
{4(\eta^2 + |z|^2)} \ . 
\]
We can now find a lower bound to $w_1$: 
\begin{eqnarray*}
  w_1&=& \frac 1n \tr D T T^* \ = \ \frac 1n \sum_{i=1}^N d_i \sum_{j=1}^N |T_{ij}|^2 \ 
=\ \frac 1n \tr D \sum_{i=1}^N \kappa_i \sum_{j=1}^N |T_{ij}|^2
\qquad \textrm{with}\quad \kappa_i= \frac{d_i}{\tr D} \\
&\stackrel{(a)}\ge& \frac 1n \tr D \left( \sum_{i=1}^N \kappa_i \left( \sum_{j=1}^N |T_{ij}|^2 \right)^{1/2} \right)^2
\ \ge\ \frac 1n \tr D \left( \sum_{i=1}^N \kappa_i |T_{ii}| \right)^2
\ \ge\  \frac 1n \tr D \left| \sum_{i=1}^N \kappa_i T_{ii} \right|^2\\
&=& \frac{|\delta|^2}{\frac 1n \tr D} 
\ \ge  \ 
 \frac{({\bs\delta}_z \boldsymbol{\ell^-}\boldsymbol{d_{\min}})^2 }
{ 16 \ \boldsymbol{\ell^+}\boldsymbol{d_{\max}} 
(\eta^2 + |z|^2)^2} 
\end{eqnarray*}
where $(a)$ follows by convexity. A similar computation yields 
$\tilde w_1 \ge ({\bs\delta}_z  \bs{\tilde d_{\min}}) ^2  
/ (16 \ \boldsymbol{\tilde d_{\max}} 
(\tilde\eta^2 + |z|^2)^2)$ where $\tilde\eta$ is a positive constant. 
Grouping these estimates with those in \eqref{eq-bound-Imz/delta} and 
plugging them into \eqref{eq-det(I-C1)}, we obtain
\begin{align*}
\det(I - C_1) &\ge 
\frac{{\bs\delta}_z^8 \ (\boldsymbol{\ell^-} \boldsymbol{\tilde d_{\min}}
\boldsymbol{\tilde d_{\min}})^2}
{
256 \ 
(\boldsymbol{\ell^+} \boldsymbol{d_{\max}} \boldsymbol{\tilde d_{\max}} )^2 
(\eta^2 + |z|^2)^2 (\tilde\eta^2 + |z|^2)^2  
} \\ 
&\ge  K\frac{{\bs\delta}_z^8}
{ (\max(\eta,\tilde\eta)^2 + |z|^2)^4 }
\end{align*} 
where $K$ is a nice constant. The same bound holds for $z \in (-\infty,0)$ by continuity of 
$\det(I - C_1(z))$ at any point of the open real negative axis. 

\subsection*{Proof of Proposition \ref{prop:auxiliary-det}-(ii)}
Recall that 
$$
{\bs \varepsilon}_n = \frac 1n \mathrm{Tr}\, D (\E Q-R)\ . 
$$
We first establish useful estimates.
\begin{lemma}
\label{lm-bound-epsilon} 
There exists nice polynomials $\Phi$ and $\Psi$ such that:
\[
\left| \frac{\Imm({\bs\varepsilon}_n(z))}{\Imm(z)} \right| 
\leq 
\frac 1n \Phi(|z|) \Psi\left(\frac{1}{{\bs\delta}_z}\right) 
\ \text{and} \ 
\left| \frac{\Imm(z{\bs\varepsilon}_n(z))}{\Imm(z)} \right| 
\leq 
\frac 1n \Phi(|z|) \Psi\left(\frac{1}{{\bs\delta}_z}\right) \ 
\ \text{for } z \in \C - \R^+ \ . 
\]
\end{lemma}
\begin{proof} 
We prove the first inequality. 
By Proposition \eqref{prop:intermediate2}-(ii), the sequence of 
functions $({\bs \varepsilon}_n)$ satisfies over $\C - \R_+$ 
\[
| {\bs \varepsilon}_n(z) | \leq 
   \frac 1n   \Phi(|z|) \Psi\left( \frac{1}{\dzz} \right) 
\]
where $\Phi$ and $\Psi$ are nice polynomials. Let ${\mathcal R}$ be the
region of the complex plane defined as 
${\mathcal R} = \{ z \ : \ \Real(z) < 0, \ 
| \Imm(z) | < - \Real(z) / 2 \}$. If $z \in \C - \R^+ - {\mathcal R}$, 
then $|\Imm(z)| \geq \dzz / \sqrt{5}$, therefore 
$| \Imm{\bs\varepsilon}(z) / \Imm z | \leq n^{-1} \sqrt{5} \dzz^{-1} 
\Phi(|z|) \Psi(\dzz^{-1})$ and the result is proven. 
Assume now that $z \in {\mathcal R}$. 
In this case, $z$ belongs to the open disc ${\mathcal D}_z$ centered at 
$\Real(z)$ with radius $- \Real(z) / 2$. 
For any $u \in {\mathcal D}_z$, we have $| {\bs\varepsilon}(u) | \leq n^{-1} 
\Phi(|u|) \Psi(|u|^{-1})$. Moreover, 
\[
\forall u \in {\mathcal D}_z, \quad 
\frac{\dzz}{\sqrt{5}} \leq -\frac{\Real(z)}{2} 
\leq |u| \leq -\frac{3\Real(z)}{2} \leq \frac{3 |z|}{2} \ .
\] 
As $\Phi(x)$ is increasing and $\Psi(1/x)$ is decreasing in $x > 0$, we 
obtain:
\begin{equation}\label{eq:estimate-epsilon}
| {\bs \varepsilon}_n(u) | \leq \frac 1n 
     \Phi\left( \frac{3 |z| }{2} \right) 
\Psi\left( \frac{\sqrt{5}}{\dzz} \right) 
\quad \text{for \ } u \in {\mathcal D}_z \ . 
\end{equation}
The function ${\bs\varepsilon}$ is holomorphic on ${\mathcal D}_z$. 
Consider the function:
Applying
Lemma \ref{lm-schwarz} with 
\[
f(\zeta) = \frac{{\bs\varepsilon}\left(|\Real(z)/2| \zeta + \Real(z)\right) - 
{\bs\varepsilon}(\Real(z))}{\sup_{u\in {\mathcal D}_z} 
| {\bs\varepsilon}(u) - {\bs\varepsilon}(\Real(z)) |}\ .  
\]
Let $\zeta= \mathbf{i} 2\Imm(z)/ \Real(z)$, apply Lemma \ref{lm-schwarz}, and use \eqref{eq:estimate-epsilon}. This yields:
\[
\left| {\bs\varepsilon}(z) - {\bs\varepsilon}(\Real(z)) \right| 
\leq \frac{2| \Imm(z) |}{|\Real(z)|}\times \frac 1{n}  
\Phi\left( |z| \right) \Psi\left( \frac{1}{\dzz} \right)
\leq \frac{\sqrt{5}| \Imm(z) |}{\dzz} \times \frac 1n
\Phi\left( |z| \right) \Psi\left( \frac{1}{\dzz} \right)\ ,
\]
where $\Phi$ and $\Psi$ are nice polynomials. As 
$\Imm({\bs\varepsilon}(\Real(z))) = 0$, we obtain
\[
\left| \frac{\Imm({\bs\varepsilon}_n(z))}{\Imm(z)} \right| 
\leq 
\left| \frac{{\bs\varepsilon}(z) - {\bs\varepsilon}(\Real(z))}
{\Imm(z)} \right| \leq 
\frac{\sqrt{5}}{\dzz n}  
\Phi\left( |z| \right) \Psi\left( \frac{1}{\dzz} \right)\ .
\]
This proves the first inequality. The second one can be proved similarly. 
\end{proof} 

We now tackle the proof of Proposition \ref{prop:auxiliary-det}-(ii), following
closely the line of the proof of Proposition \ref{prop:auxiliary-det}-(i). 
Recall that
$\alpha=\frac 1n \tr D\E Q$, $\tilde \alpha= \frac 1n \tr \tilde D
\E \tilde Q$, $\boldsymbol{\varepsilon} = \frac 1n \tr D(\E Q - R)$, and 
$\boldsymbol{\tilde\varepsilon} = \frac 1n \tr \tilde D(\E \tilde Q - 
\tilde R)$. We begin by establishing the lower bound on $\det(I-C_2)$. 
Assume that $z \in \C^+ \cup \C^-$. Writing $\alpha = \frac 1n \tr DR +
\boldsymbol{\varepsilon}$ and $\tilde \alpha = \frac 1n \tr\tilde
D\tilde R + \boldsymbol{\tilde \varepsilon}$ and developing 
$\Imm(\alpha)$ and $\Imm(z\tilde\alpha)$ with the help of the resolvent 
identity, we get the following system:
\[
(I - C_2)  
\left(
\begin{array}{c}
\Imm(\alpha)\\
\Imm(z\tilde \alpha)
\end{array}
\right) 
= 
\Imm(z)\ \left(
\begin{array}{c}
w_2(z)\\
\tilde x_2(z)
\end{array}
\right)
+ 
\left(
\begin{array}{c}
\Imm(\boldsymbol{\varepsilon})\\
\Imm(z \boldsymbol{\tilde \varepsilon})
\end{array}
\right) \ ,
\] 
where $w_2(z) = \frac 1n \tr DR R^*$ and $\tilde x_2(z) >0$. Let
$\tilde w_2= n^{-1} \tr \tilde D \tilde R \tilde R^* $. Using the same
arguments as in the proof of Proposition \ref{prop:auxiliary-det}-(i),
we obtain
\begin{eqnarray}
\label{eq-(1-u2)} 
1 - u_2 &= &w_2 \frac{\Imm(z)}{\Imm(\alpha)} 
+ {v}_2 \frac{\Imm(z\tilde\alpha)}{\Imm(\alpha)} + 
\frac{\Imm( \boldsymbol{\varepsilon})}{\Imm(\alpha)} 
\ge
 w_2 \frac{\Imm(z)}{\Imm(\alpha)} + 
\frac{\Imm( \boldsymbol{\varepsilon})}{\Imm(\alpha)} \ ,\\
\label{eq-(1-tilde-u2)} 
1 - \tilde u_2 &=& \tilde w_2 \frac{\Imm(z)}{\Imm(\tilde\alpha)} 
+ \tilde{v}_2 \frac{\Imm(z\alpha)}{\Imm(\tilde\alpha)} + 
\frac{\Imm( \boldsymbol{\tilde \varepsilon})}{\Imm(\tilde\alpha)} 
\ge
\tilde w_2 \frac{\Imm(z)}{\Imm(\tilde\alpha)} + 
\frac{\Imm( \boldsymbol{\tilde \varepsilon})}{\Imm(\tilde\alpha)} \ , \\
\det(I-C_2) 
&\geq &
w_2 \tilde w_2 \frac{\Imm(z)}{\Imm(\alpha)} \frac{\Imm(z)}{\Imm(\tilde\alpha)} 
+ 
w_2 \frac{\Imm(z)}{\Imm(\alpha)} 
\frac{\Imm( \boldsymbol{\tilde \varepsilon})}{\Imm(\tilde\alpha)} 
+ (1 - \tilde u_2) \frac{\Imm({\bs\varepsilon})}{\Imm(\alpha)} 
+ v_2 \frac{\Imm(z \bs{\tilde\varepsilon})}{\Imm(\alpha)} \nonumber \\
&\stackrel{\triangle}=& 
w_2 \tilde w_2 \frac{\Imm(z)}{\Imm(\alpha)} \frac{\Imm(z)}{\Imm(\tilde\alpha)} 
+ e(z)\ .
\label{eq-det(I-C2)} 
\end{eqnarray}

We now find an upper bound on the perturbation term $e(z)$. 
To this end, we have 
$0 \leq w_2 \leq {\boldsymbol{\ell^+} \boldsymbol{d_{\max}}}
/ {{\bs\delta}_z^2}$ and 
$0 \leq v_2 \leq  {\boldsymbol{\ell^+} \boldsymbol{d}_{\bs\max}^2} / 
 {\boldsymbol{\delta}_z^2}$.   
% \quad 
% 0 \leq \tilde v_2 \leq \frac{\bs{\tilde d}_{\bs\max}^2}
% {{\boldsymbol\delta}_z^2}.  
Recalling \eqref{eq-C2}, we also have
\[
| 1 - \tilde u_2 | \leq 
1 + \frac{\boldsymbol{d_{\max}} \boldsymbol{\tilde d_{\max}} 
{\bs a}^2_{\bs\max} |z|^2}{{\bs \delta}_z^4} \ . 
\]
Using the same arguments as in the proof of Proposition
\ref{prop:auxiliary-det}-(i) (involving this time the tightness of the
measures associated with the Stieltjes transforms $\frac 1n \tr DR$
and $\frac 1n \tr \tilde D \tilde R$) yields:
\begin{eqnarray*}
\frac{\Imm(z)}{\Imm(\alpha)} 
\le
\frac{4(\eta^2 + |z|^2)}{\boldsymbol{\ell^-}\boldsymbol{d_{\min}} } \ ,\quad
|e(z)| \leq \frac 1n  \Phi(|z|) \Psi(\dzz^{-1})\ ,\quad \frac{\Imm(z)}{\Imm(\alpha)} \le
\frac{4(\eta^2 + |z|^2)}{\boldsymbol{\ell^-}\boldsymbol{d_{\min}} } \ ,
\end{eqnarray*}
for every $z \in \C^+ \cup \C^-$, where $\eta,K$ are positive constants, and $\Phi$ and $\Psi$, nice polynomials.

% . Combining these estimates with the result of Lemma \ref{lm-bound-epsilon}, 
% we obtain that $|e(z)| \leq n^{-1} \Phi(|z|) \Psi(\dzz^{-1})$ where 
% $\Phi$ and $\Psi$ are nice polynomials.

% It has been proven in \cite[Lemma C.1]{HLN07} 
% that the sequence of positive measures $(\mu_n)$ with total mass
% $n^{-1} \tr (D)$ is tight. In these conditions, 
% a computation similar to \eqref{eq-inf-delta} shows that for every
% $z \in \C^+ \cup \C^-$, 
% \[
% \frac{\Imm(z)}{\Imm(\alpha)} 
% \le
% \frac{4(\eta^2 + |z|^2)}{\boldsymbol{\ell^-}\boldsymbol{d_{\min}} } \ ,
% \]
% where $\eta$ is a positive constant. Combining these estimates with the result of Lemma \ref{lm-bound-epsilon}, 
% we obtain that $|e(z)| \leq n^{-1} \Phi(|z|) \Psi(\dzz^{-1})$ where 
% $\Phi$ and $\Psi$ are nice polynomials.

% Combining these estimates with the result of Lemma \ref{lm-bound-epsilon}, 
% we obtain that $|e(z)| \leq n^{-1} \Phi(|z|) \Psi(\dzz^{-1})$ where 
% $\Phi$ and $\Psi$ are nice polynomials. Considering now the first term at the
% right hand side of \eqref{eq-det(I-C2)}, an argument similar to the one
% made in the proof of Proposition \ref{prop:auxiliary-det}-(i) 
% (involving this time the tightness of the measures associated with the Stieltjes transforms 
% $\frac 1n \tr DR$ and $\frac 1n \tr \tilde D \tilde R$) shows that
% \[
% w_2 \tilde w_2 \frac{\Imm(z)}{\Imm(\alpha)} \frac{\Imm(z)}{\Imm(\tilde\alpha)} 
% \ge  K\frac{{\bs\delta}_z^8}
% { (\eta^2 + |z|^2)^4 }\ ,
% \]
% where $K$ and $\eta$ are nice constants. 
Finally, we can state that there exist nice polynomials 
$\Phi$ and $\Psi$ such that:
$$
\det(I-C_2)\ \ge  K \frac{\dzz^8 }{ (\eta^2 + |z|^2)^4 } 
\left( 1 - \frac 1n \Phi(|z|)\Psi\left( \frac 1{\dzz}\right) \right)\ .
$$
By continuity of $\det(I-C_2(z))$ at any point of the open real negative 
axis, this inequality is true for any $z \in \C - \R^+$. 
Denote by ${\mathcal E}_n$ the set: 
$$
{\mathcal E}_n=\left\{z\in \C-\R^+, \quad
 \frac 1n  \Phi(|z|)\Psi\left(\frac{1}{\dzz} \right)\le 1/2\right\}\ .
$$ 
If $z\in {\mathcal E}_n$, then $\det(I-C_2)$ is readily lower-bounded by
the quantity stated in Proposition \ref{prop:auxiliary-det}-(ii). \\
By considering inequalities \eqref{eq-(1-u2)} and \eqref{eq-(1-tilde-u2)}
and by possibly modifying the polynomials $\Phi$ and $\Psi$, 
we have $1 - u_2 \ge 0$ and $1 - \tilde u_2 \ge 0$ for $z \in {\mathcal E}_n$. 
The proof of proposition \ref{prop:auxiliary-det}-(ii) is completed. 
\end{appendix}

\bibliography{math}

\noindent {\sc Walid Hachem} and {\sc Jamal Najim},\\ 
CNRS, T\'el\'ecom Paristech\\ 
46, rue Barrault, 75013 Paris, France.\\
e-mail: \{hachem, najim\}@telecom-paristech.fr\\
\\
\noindent {\sc Philippe Loubaton} and {\sc Pascal Vallet},\\
Institut Gaspard Monge LabInfo, UMR 8049\\
Universit\'e Paris Est Marne-la-Vall\'ee\\
5, Boulevard Descartes,\\
Champs sur Marne,
77454 Marne-la-Vall\'ee Cedex 2, France\\
e-mail: \{philippe.loubaton, pascal.vallet\}@univ-mlv.fr\\
\noindent

\end{document}